\documentclass[10pt,reqno]{amsart}

\usepackage{graphicx} 
\usepackage{amsmath, amsthm, amssymb,indentfirst}

\usepackage{amsmath, amsthm, amsopn, amssymb, enumerate}

\usepackage{amsmath , amssymb, latexsym }
\usepackage{graphics}
\usepackage{graphicx}
\usepackage{verbatim} 
\usepackage{color}
\usepackage[mathscr]{eucal}

\usepackage{pgf,tikz}
\usepackage{mathrsfs}
\usetikzlibrary{arrows}
\usetikzlibrary{decorations.pathmorphing, decorations.markings, calc, cd}
\usepackage[font=small,labelfont=bf]{caption}
\usepackage{mathtools}
\usepackage[all]{xy}

\usepackage{hyperref}
\theoremstyle{plain}
\newtheorem{theorem}{Theorem}[section]
\newtheorem{corollary}[theorem]{Corollary}
\newtheorem*{corollary*}{Corollary}
\newtheorem{prop}[theorem]{Proposition}

\newtheorem{lemma}[theorem]{Lemma}
\newtheorem*{proposition*}{Proposition}
\newtheorem*{theorem*}{Theorem}
\newtheorem*{lemma*}{Lemma}

\newtheorem*{claim*}{Claim}
\theoremstyle{definition}
\newtheorem{Cons}[theorem]{Construction}
\newtheorem{definition}[theorem]{Definition}
\newtheorem*{definition*}{Definition}
\newtheorem{remark}[theorem]{Remark}

\newtheorem*{obs*}{Observation}

\def\ZZ{\mathbb{Z}}

\newcommand{\ora}[1]{\overrightarrow{#1}}
\newcommand{\ol}[1]{\overline{#1}}

\newcommand{\eq}[1]{\begin{equation}\label{eq:#1}}
\newcommand{\eqe}{\end{equation}}

\newcommand{\net}{\mathfrak{F}}

\def\C{\mathcal{C}}
\def\D{\mathcal{D}}
\def\F{\mathcal{F}}

\def\G{\Gamma}

\def\F{\mathcal{F}}
\def\M{\mathcal{M}}

\def\R{\mathcal{R}}

\def\Y{\mathcal{Y}}

\def\RR{\mathbb{R}}
\def\ZZ{\mathbb{Z}}

\def\vr{\varphi}
\def\O{\mathcal{O}}

\def\ge{\geqslant}

\def\geq{\geqslant}

\def\<{\langle}
\def\>{\rangle}
\def\0{\textbf{0}}

\def\col{\colon}

\DeclareMathOperator{\Aut}{Aut}
\DeclareMathOperator{\Div}{Div}
\DeclareMathOperator{\divi}{div}
\DeclareMathOperator{\rank}{rank}
\DeclareMathOperator{\trop}{trop}
\DeclareMathOperator{\val}{val}

\DeclareMathOperator{\roots}{Root}

\DeclareMathOperator{\Eff}{Eff}
\DeclareMathOperator{\Spec}{Spec}

\DeclareMathOperator{\sing}{sing}
\DeclareMathOperator{\mon}{Mon}

\DeclareMathOperator{\Supp}{Supp}

\DeclareMathOperator{\Hom}{Hom}
\DeclareMathOperator{\eff}{Eff}
\DeclareMathOperator{\Roots}{Root}

\DeclareMathOperator{\an}{an}
\DeclareMathOperator{\et}{et}

\newcommand{\flow}{\mathbf{F}}
\newcommand{\flowp}{\mathcal{F}}
\newcommand{\lra}{\longrightarrow}
\newcommand{\ra}{\rightarrow}

\tikzset{snake it/.style={decorate, decoration=snake}}

\begin{document}

\setcounter{tocdepth}{1}

\title[On moduli spaces of roots in algebraic and tropical geometry]{On moduli spaces of roots in algebraic and tropical geometry} 
\author{Alex Abreu}
\author{Marco Pacini}
\author{Matheus Secco}

\thanks{The second author is supported by CNPq, bolsa PQ 301671/2019-2, and Faperj CNE E-26/200.845/2021 (259616)}
\thanks{The work was partially done while the third author was affiliated with Institute of Computer Science of the Czech Academy of Sciences, with institutional support RVO:67985807}

\email{}

\maketitle

\begin{abstract}
In this paper we construct a tropical moduli space parametrizing roots of divisors on tropical curves. We study the relation between this space and the skeleton of Jarvis moduli space of nets of limit roots on stable curves. We show that the combinatorics of the moduli space of tropical roots is governed by the poset of flows, a poset parametrizing certain flows on graphs. 
\end{abstract}

\bigskip



\tableofcontents

\section{Introduction}

The moduli space of $r$-roots of a universal invertible sheaf on  smooth $n$-pointed curves is a natural unramified covering of $\mathcal M_{g,n}$. The construction of a  compactification of this space over $\overline{\mathcal M}_{g,n}$ has attracted a lot of interest in the last three decades and has found several applications. To our knowledge, the first such  compactification was given by Cornalba in \cite{CornalbaTheta}. He constructed the moduli space of spin curves, which takes care of the case where $r=2$ and the invertible sheaf is the relative canonical sheaf. This was continued with the work of Jarvis \cite{Jarvis} and \cite{JarvisNet}, containing a construction of a compactification for any $r$ and any sheaf via torsion-free rank-$1$  sheaves. Then Abramovich and Jarvis \cite{AJ} and  Abramovich, Corti, Vistoli \cite{ACV} solved the problem for any $r$ using twisted bundles, and Caporaso, Casagrande and Cornalba \cite{CCC} construced a compactification for any $r$ via invertible sheaves on semistable curves. Finally, Chiodo \cite{Chiodo} constructed a compatification for any $r$ via twisted curves with the property that the natural map to (a suitable modification of)  $\overline{\M}_{g,n}$ is unramified. The compactifications are a fundamental tool in the computation of the double ramification cycle (see for example \cite{JPPZ}). For a recent generalization of the work of Chiodo and Jarvis with applications to the double ramification cycle, see \cite{HO}.  

In this paper, we mainly consider the Jarvis compactification \cite{JarvisNet}. This is the space paramatrizing nets of $r$-limit roots, and it is a normalization of the space constructed by Jarvis in \cite{Jarvis}, or by Caporaso, Casagrande and Cornalba in \cite{CCC} (the last two spaces are isomorphic). 

Since the work of \cite{ACP}, it is natural to construct tropical analogues of moduli spaces arising in algebraic geometry. These tropical analogues encode the combinatorial data of the objets parametrized by the algebro-geometric spaces. In some cases they even recover topological information from the algebro-geometric space, as in \cite{CGP}.
The tropical spaces are often isomorphic to the skeleton (in the sense of Berkovich) of the analytification of these moduli spaces. The relation between algebraic geometry and tropical geometry is made more concrete with the work of \cite{CCUW} that constructs Artin stacks from tropical moduli spaces.
 Several such tropical analogues have been constructed:  the moduli space of tropical curves in \cite{Mikh} and \cite{BMV}, the  moduli space of tropical admissible covers in \cite{CHMR}, the universal tropical Jacobian in \cite{APPLMS} and \cite{MMUV}, and the moduli space of tropical spin curves in \cite{CMP}.

In this paper we consider roots of divisors on tropical curves. The first description of theta characteristics of tropical curves, namely square roots of the canonical divisor of a tropical curve, was given in \cite{MZ} and \cite{Zharkov}. Our first goal is to describe all $r$-roots of any divisor on tropical curves. Our construction is compatible with specializations of graphs. This allows us to construct a tropical analogue of the space of $r$-roots which is the target of a tropicalization map from the analytification of the Jarvis moduli space.
This space is a generalized cone complex over the moduli space of tropical curves and parametrizes roots of a fixed universal divisor on tropical curves. It is not isomorphic to the skeleton of its algebro-geometric counterpart, but it is the target of a natural map from the skeleton that commutes with the tropicalization map from the analytification of Jarvis space. In fact, to fully describe the skeleton of the space of $r$-roots of an invertible sheaf, an extra-datum is necessary. This is done in \cite{CMP} in the case of spin curves, where the extra-datum is a sign function, encoding the parity of the space of sections of the invertible sheaves. Other special cases should be doable, as for instance the case of the moduli space of Prym varieties. The general case seems quite difficult and  challenging.

\subsection{The results}

The main combinatorial part of this article concerns graphs and flows on graphs. Let $g,n$ be nonnegative integers satifying $2g-2+n > 0$.  As usual we denote by $\mathcal G_{g,n}$ the poset of stable genus-$g$ graphs with $n$ legs.  Given a ramification sequence $\R=(a_1,\ldots, a_n,\ell)$ with elements in an abelian group $A$, we construct a poset $\flowp_{g}(A,\R)$, which we call the \emph{poset of flows}, consisting of pairs $(\Gamma,\vr)$, where $\Gamma$ is a genus-$g$ stable graph with $n$-legs incident to the vertices $v_1,\dots,v_n$ of $\G$, and $\vr$ is a flow on $\Gamma$ with values in $A$ whose associated divisor  $\divi(\vr)$ satisfies the condition  $\divi(\vr)=a_1v_1+\ldots+a_nv_n+\ell K_{\Gamma}$ (here $K_{\G}$ is the canonical divisor of $\G$).

\begin{theorem*}[Theorems \ref{thm:main1}]
The poset $\flowp_{g}(A,\R)$ satisfies the following properties.
\begin{enumerate}[(i)]
    \item It has a unique minimal element which is the pair $(\G,\varphi)$ where $\G$ has only one vertex with weight $g$ and no edges, and $\varphi$ is the trivial flow.
    \item Its maximal elements  are the pairs $(\G,\varphi)$, where $\G$ is a $3$-regular stable graph with $n$-legs.
    \item It is connected and ranked of dimension $3g-3+n$, where the rank of a pair $(\G,\varphi)$ is given by the number of edges of $\G$.
    \item There is a forgetful surjective map of ranked posets $\flowp_{g}(A,\R)\to \mathcal G_{g,n}$. 
    For $\R=(0,\ldots, 0)$, the fiber over the class of a graph $\Gamma$ is isomorphic to $H_1(\Gamma,A)/\Aut(\Gamma)$.
\end{enumerate}
\end{theorem*}



In fact, the poset $\flowp_{g}(A,\R)$ is connected through codimension one: we will prove this fact in the forthcoming paper \cite{APS2}.

The poset of flows $\flowp_{g}(A,\R)$ is the poset underlying the tropical moduli space parametrizing roots of divisors on tropical curves. For a ramification sequence $\R=(a_1,\ldots, a_n,\ell)$ and a $n$-pointed stable genus-$g$ tropical curve $X$, we write $\D_{X,\R}:=a_1p_1+\ldots +a_np_n+\ell K_X$, where $p_i$ are the marked points on $X$ and $K_X$ is the canonical divisor of $X$.

Let $\Roots_{g}^{r,\trop}(A,\R)$ be the generalized cone complex defined as 
\[
\Roots_{g}^{r,\trop}(A,\R)=\bigsqcup \mathbb{R}^{E(\Gamma)}_{>0}/\Aut(\Gamma,\vr),
\]
where the union is taken over pairs $(\G, \vr)$ in the poset $\flowp_g(A/rA,\overline{\mathcal R})$ (here, $\overline{\mathcal R}$ is the ramification sequence $\R$ with entries  taken modulo $A$).
There is a natural forgetful map
\[
\pi_{g,\R}^{\trop}\col \Roots^{r,\trop}_{g}(A,\R)\to M^{\trop}_{g,n}.
\]
\begin{theorem*}[Theorem \ref{thm:main-moduliroot}]
   The following properties hold.
   \begin{enumerate}[(i)]
       \item The points of $\Roots^{r,\trop}_g(A,\R)$ parametrize pairs $(X,[\D])$, where $X$ is a genus-$g$ stable tropical curve with $n$ legs, and $\D$ is an $r$-root of $\D_{X,\R}$. 
       \item The map $\pi_{g,\mathcal R}^{\trop}$ is a map of generalized cone complexes  whose fiber over the class of a tropical curve $X$ in $M_{g,n}^{\trop}$ is in bijection with the set of $r$-roots of $\D_{X,\R}$, modulo $\Aut(X)$.
   \end{enumerate}
  \end{theorem*}
 
 The bijection in item (ii) of the above theorem is explicit and given in Definition \ref{def:divisor_flow}. This allow us to give an explicit one-to-one correspondence between flows in $\flowp_g(A/rA,\overline{\mathcal R})$ and roots of $\D_{X,\R}$, see Theorem \ref{thm:Delta_Phi}.

Next, we come back to algebraic geometry. Let $\overline{\M}_{g,n}^r(\R)$ be the Jarvis  moduli space of nets of $r$-limit roots and $\overline{\Sigma}(\overline{\M}_{g,n}^r(\R)^{\an})$ be its skeleton (in the sense of Berkovich). Set $\Roots^{r,\trop}_g(\R):=\Roots^{r,\trop}_g(\mathbb Z,\R)$.
The natural compactification  $\overline{\Roots}^{r,\trop}_{g}(\R)$ of $\Roots^{r,\trop}_g(\R)$ is related to the analytication $\overline{\M}_{g,n}^r(\R)^{\an}$ of the space $\overline{\M}_{g,n}^r(\R)$ via a tropicalization map (see Equation \ref{eq:tropmap})
\[
\trop_{\ol{\M}_{g,n}(\R)}\col \overline{\M}_{g,n}^r(\R)^{\an}\to \overline{\Roots}^{r,\trop}_g(\R),
\]
which factors through a map $\Phi_{\ol{\M}_{g,n}^{r}(\R)}\col \overline{\Sigma}(\overline{\M}_{g,n}^r(\R))\to \overline{\Roots}^{r,\trop}_g(\R)$.

\begin{theorem*}[Theorem \ref{thm:tropj}]
There is a morphism of extended generalized cone complexes
\[
\Phi_{\ol{\M}_{g,n}^{r}(\R)}\col \ol{\Sigma}(\ol{\M}_{g}^{r}(A,\R))\to \ol{\Roots}_{g}^{r,\trop}(\R)
\]
which makes the following diagram commute
\begin{eqnarray*}
\SelectTips{cm}{11}
\begin{xy} <16pt,0pt>:
\xymatrix{
\ol{\M}_{g,n}^r(\R)^{\an} \ar@/^2pc/[rrrr]^{\trop_{\ol{\M}_{g,n}^r(\R)}} \ar[d]_{\pi^{\an}_{g,\mathcal R}} \ar[rr]^{{\bf p}_{\ol{\M}_{g,n}^r(\R)}\;}   
  && \ar[rr]^{{\Phi}_{\ol{\M}_{g,n}^r(\R)}} \ar[d]_{\ol{\Sigma}(\pi^{\an}_{g,\mathcal R})} \ol{\Sigma}(\ol{\M}_{g,n}^r(\R)) && \ol{\Roots}_{g}^{r,\trop}(\R)   \ar[d]_{\pi^{\trop}_{g,\mathcal R}} \\
\ol{\M}_{g,n}^{\an}  \ar@/_2pc/[rrrr]_{\trop_{\ol{\M}_{g,n}}} \ar[rr]^{{\bf p}_{\ol{ {\M}}_{g,n}}\;}            & &  \ar[rr]^{{\Phi}_{\ol{\M}_{g,n}}}         \ol{\Sigma}(\ol{{\M}}_{g,n})  && \ol{M}_{g,n}^{trop}  
 }
\end{xy}
\end{eqnarray*}
(In the diagram $\bf p$ denotes the natural retraction map onto the skeleton.)
\end{theorem*}

 In \cite{APS3} we compare the tropical space $\Roots_{g}^{r,\trop}(\R)$ with the  moduli space of tropical admissible covers 
 constructed in \cite{CHMR}.  

\subsection*{Acknowledgements} 

We thank Lucia Caporaso, Ethan Cotterill, Margarida Melo, Martin Ulirsch for many important discussions on the topics of the paper.

\section{Background material}

\subsection{Graphs}
\label{sec:graphs}

Let $\G = (V(\G), E(\G))$ be a graph. We will only consider connected graphs, while we will allow possible disconnected subgraphs of a graph. Given a subset $V\subset V(\G)$, we set $V^c:=V(\G)\setminus V$. Given $v\in V(\G)$ we denote by $\val_{\G}(v)$ the valence of $v$ (with loops counting twice).
The set of \emph{oriented edges} of $\G$ is defined as $\ora{E}(\G):=E(\G)\sqcup E(\G)$. We have a map $\ora{E}(\G)\ra E(\G)$ given by $e\mapsto e$. 
If $e\in \ora{E}(\Gamma)$, we will denote by $-e\in \ora{E}(\Gamma)$ the oriented edge, distinct from $e$, such that $e$ and $-e$ have the same underlying edge  
(we can view $e$ and $-e$ as the two possible orientations on a given edge).
We let $\sigma, \tau \col \ora{E}(\G)\to V(\G)$ be the source and target map, so that $\sigma(e)$ and $\tau(e)$ are the vertices incident to $e$, for every $e\in \ora{E}(\G)$. An \emph{orientation} $\ora{\G}$ on $\G$ is a section of the function $\ora{E}(\G)\to E(\G)$; we denote by $E(\ora{\G})$ the image of this section.

We denote by $(\G,w_\G,L_\G)$ a genus-$g$ stable graph with $n$-legs, where $w_\G\col V(\G)\to \mathbb{Z}_{\geq0}$ is the weight function and $L_\G\col [n]\to V(\G)$ is the leg function. We will usually write $\G$ instead of $(\G,w_{\G},L_{\G})$. The \emph{genus} of $\G$ is $g=:g_\G:=b_1(\G)+\sum_{v\in V(\G)}w(v)$, where $b_1(\G)$ is the first Betti number of $\G$. 
We say that $\G$ is \emph{stable} if $2\val_{\G}(v)-2+w_{\G}(v)+L^{-1}_\G(v)>0$, for every $v\in V(\G)$. In  particular, if $\G$ is a genus-$g$ stable graph with $n$ legs, then $2g-2+n>0$. For a graph $\Gamma$ and a subset $\mathcal E\subset E(\Gamma)$ we denote by $\Gamma/\mathcal E$ the graph obtained by contracting all the edges $e\in \mathcal E$.
We refer to \cite{APPLMS} for the notions and properties of specializations of weighted graphs with legs (that is, weighted graphs obtained by contracting edges of a graph).


We say that $\G$ is a \emph{tree} if $b_1(\G)=0$, and \emph{tree-like} if $\G$ becomes a tree when contracting its loops. 
A \emph{refinement} $\G'$ of $\G$ is a graph whose edges are obtained by refining the edges of $\G$, i.e., inserting a certain number $n_e$ (possibly zero) of vertices in the interior of each edge $e$ of $\G$. We say that an edge of $\G'$ is \emph{over} and edge $e$ of $\G$ if it is obtained by refining $e$. Notice that any orientation of an edge $e$ induces a natural orientation on every edge over $e$.

Let $A$ be an Abelian group. 
An \emph{$A$-divisor} on $\G$ is a function $D\col V(\G)\to A$. The \emph{degree} of an $A$-divisor $D$ is  $\deg(D)=\sum_{v\in V(\G)} D(v)\in A$. We denote by $\Div(\G,A)$ be the Abelian group of $A$-divisors on $\G$, and for $a\in A$, we let by $\Div^a(\G,A)$ the subset of divisors of degree $a$.

\subsection{Ranked posets}

Let  $S=(S, \geq)$ be a poset. We say that $S$ is \emph{ranked} if every maximal chain in $S$ has the same length. A ranked poset has a rank function $\rank\col S\to \mathbb{Z}_{\geq0}$ such that $\rank(x)=\rank(y)+1$ if $x$ covers $y$ and $\rank(x)=0$ if $x$ is minimal. In particular, every maximal element has the same rank, which is the \emph{dimension} of $S$. A map between ranked posets is an order and rank preserving function.



We let $\mathbf{G}_{g,n}$  be the category whose objects are genus-$g$ stable graphs with $n$-legs and whose morphisms are specializations. We let $\mathcal{G}_{g,n}$ be the poset whose elements are the isomorphisms classes of elements of $\mathbf{G}_{g,n}$. The poset $\mathcal{G}_{g,n}$ is ranked with rank function given by $\rank(\G)=|E(\G)|$.


\subsection{Flows on graphs}

 Let $A$ be an Abelian group $A$ with identity $0$ (in this paper we usually use $A=\ZZ$ or $A=\ZZ/m\ZZ$). For elements $a\in A$, $b\in \mathbb{Z}$, we define 
\[
\gcd(b, a + bA ) = \max\{j\in \mathbb{Z}; j \text{ divides } |b|, a+bA\subset jA \}.
\]
Notice that, when $A=\mathbb Z$, we have that $\gcd(b, a + bA )$ is the usual $\gcd(b,a)$. If $d\in \mathbb{Z}$ is such that the multiplication by $d$ is injective in the group $A$, we have that
\begin{equation}
    \label{eq:d_gcd}
    \gcd(db, da+dbA)= d\gcd(b, a+bA).
\end{equation}
We will also need the following lemma.
\begin{lemma}
\label{lem:torsion_quotient}
Let $A$ be an abelian group and $r$ be an integer such that the multiplication map $A\xrightarrow{\cdot r} A$ is injective. If $a$ is a torsion element of $A$, then $a \in rA$.
\end{lemma}

\begin{proof}
   Let $n\in \mathbb{Z}$ be such that $na = 0$ and let $d=\gcd(n,r)$. Then we have $d\frac{n}{d}a=0$. Since multiplication by $d$ is injective (because $d$ divides $r$), we have that $\frac{n}{d}a=0$. Repeating the argument we can assume that $n$ is coprime with $r$ (here we use that $n\neq 0$). If we write $nx+ry=1$, for some $x,y\in \mathbb Z$, we have that $a= (xn+ry)a=r(ya)$, which means that $a\in rA$.
\end{proof}

Let $\G$ be a graph.
An \emph{$A$-flow} is a function $\varphi\col \ora{E}(\Gamma)\to A$ such that $\varphi(e)+\varphi(-e)=0$ for every oriented edge $e\in \ora{E}(\Gamma)$. 
We denote by $C_1(\Gamma,A)$ be the $A$-module of $A$-flows on $\Gamma$. 
  Given a specialiation of graphs $\iota\col \Gamma\to \Gamma'$ and a flow $\varphi$ on $\Gamma$, we define the flow $\varphi':=\iota_*(\varphi)$ as the flow on $\G'$ such that $\varphi'(e)=\varphi(e)$, for every edge $e\in \ora{E}(\G')\subset \ora{E}(\Gamma)$. In this case, we say that $(\Gamma,\vr)$ \emph{specializes} to $(\Gamma', \vr')$.
   The \emph{divisor of an $A$-flow} $\varphi$ is defined as
  \[
  \divi(\varphi)(v):= \sum_{e\in \ora{E}(\Gamma);\tau(e)=v}\varphi(e),
  \]
  for $v\in V(\G)$.
Notice that $\divi(\varphi)$ has always degree zero.
An $A$-flow $\varphi$ is a \emph{Kirchhoff flow} if $\divi(\varphi)=0$ (the zero divisor). The following lemma is straightforward.

 \begin{lemma}\label{lem:push-flow}
 If $\iota\col \G\to \G'$ is a specialization of graphs with $n$ legs, then
for every $A$-flow $\vr$  on $\G$   we have that $\divi(\iota_*(\vr))=\iota_*(\divi(\vr))$.
    \end{lemma}

\begin{lemma}\label{lem:surjective}
Given an abelian group $A$, the map $C_1(\Gamma,A)\to \Div^0(\Gamma,A)$ given by $\varphi\to \divi(\varphi)$ is surjective.
\end{lemma}

\begin{proof}
   Let $D$ be an $A$-divisor of degree $0$ on $\Gamma$. 
   Choose a spanning tree $T$ of $\Gamma$.  Set $\varphi(e)=0$ for every $e\notin \ora{E}(T)$. Consider $e \in \ora{E}(T)$, and  let $T_1$ and $T_2$ be the components of $T$ containing the target and source of $e$, respectively. We define $\varphi(e)=\deg(D|_{T_1})- \deg(D|_{T_2})$. Then $\varphi$ is a $A$-flow on $\Gamma$ such that $\divi(\varphi)=D$.
\end{proof}

   We denote by $H_1(\Gamma, A)$ the kernel of the map
   \begin{align*}
       C_1(\Gamma,A)& \to \Div^0(\Gamma,A)\\
        \vr &\mapsto \divi(\vr).
   \end{align*}

\subsection{Tropical curves and their moduli}

Let $X$ be an $n$-pointed (weighted) tropical curve. For a model $(\Gamma,\ell)$ of $X$, we identify an edge $e\in E(\G)$ with the associated segment in $X$. Moreover, we let $e^\circ$ be the interior of $e$ in $X$.

We say that $X$ is \emph{stable} if there is a model $(\Gamma,\ell)$ of $X$ with $\G$ a (weighted) stable graph with $n$ legs. In this case, this stable graph is unique and it is called the \emph{canonical model} of $X$.

Let $A$ be an Abelian group. An \emph{$A$-divisor} on $X$ is a function $\mathcal{D}\col X\to A$ such that $\D(p)\neq 0$ for finitely many points $p\in X$. The set $\{p\in X; \mathcal{D}(p)\neq0\}$ is called the \emph{support} of $\mathcal{D}$ and it is denoted by $\Supp(\mathcal{D})$.

An \emph{$A$-rational function} on $X$ is a tuple $f=(\G_f,\ell_f,\varphi_f)$ where $(\G_f,\ell_f)$ is a model of $X$ and $\varphi_f$ is an $A$-flow on $\G$ such that 
\begin{equation}\label{eq:rationalfunction}
  \sum_{e \in \gamma} \varphi_f(e)\otimes \ell_f(e) =0,
\end{equation}
for every cycle $\gamma$ of $\G$, where the equality above is seen in $A\otimes_{\mathbb Z} \mathbb{R}$. We identify $A$-rational functions which agree in a common refinement.
Notice that if $A=\mathbb Z$, the notion of $A$-rational function coincides with the usual notion of rational function on a tropical curve.

If we have an orientation $\ora{\G}$ on $\Gamma$, the condition in Equation  \eqref{eq:rationalfunction} is equivalent to the condition
\begin{equation}\label{eq:rationalfunction-gamma}
  \sum_{e \in E(\ora{\G})}\gamma(e) (\varphi_f(e)\otimes \ell_f(e))=0,
\end{equation}
for every cycle $\gamma$ of $\G$,
where $\gamma(e) = 1$ if $e\in \gamma$, $\gamma(e) = -1$ if $-e \in \gamma$ and $\gamma(e) = 0$, otherwise.
Given an $A$-rational function $f=(\Gamma_f,\ell_f,\vr_f)$ on $\G$, the \emph{principal divisor} of $f$ is given by $\divi(\varphi_f)\in \Div(\G_f)$, seen as a divisor on $X$.

\begin{remark}\label{rem:const}
Let $f$ be a $\ZZ$-rational function on a tropical curve $X$ such that $\divi(f)=0$. Then $f$ is constant, and in particular $\vr_f=0$.
\end{remark}

\begin{lemma}\label{lem:rational-function-torsion} 
Let $f$ be $A$-rational function on a tropical curve $X$ such that $\divi(f)=0$. Then $\vr_f$ is a torsion flow.
\end{lemma}
\begin{proof}
Write $f=(\Gamma_f,\ell_f,\vr_f)$. 
   Let $A'$ be the subgroup of $A$ generated by $\{\vr_f(e) ; e\in \ora{E}(\G)\}$. By the classification of $\ZZ$-modules, we may assume that $A'=\ZZ^k\oplus\bigoplus_{i} \ZZ/k_{i}\ZZ$ for some integers $k$ and $k_i$. Let $\overline{\vr}$ be the projection of $\vr_f$ onto a $\ZZ$-summand of $A'$. Then, $\overline{\vr}$ induces a $\ZZ$-rational function on $X$ with $0$ divisor. By Remark \ref{rem:const}, we have that $\overline{\vr}$ is $0$. Hence $(\prod_{i} k_i)\vr_f=0$, which means that $\vr_f$ is a torsion flow. 
\end{proof}

\begin{lemma}
\label{lem:principal}
 Let $X$ be a tropical curve with model $(\Gamma, \ell)$. Let $e\in \ora{E}(\G)$, and consider distinct points $\sigma(e) = p_0, p_1, p_2, p_3, p_4, p_5=\tau(e)$ in $e$ such that $p_i$ is between $p_{i-1}$ and $p_{i+1}$ for $i=1,2, 3, 4$, and such that the distance between $p_1$ and $p_2$ is equal to distance between $p_3$ and $p_4$. Then, for every $a \in A$ the divisor $-ap_1 + ap_2 + ap_3 - ap_4$ is principal on $X$.
 \end{lemma}
 
 \begin{proof}
    Let $(\Gamma',\ell')$ be the model of $X$ such that $V(\Gamma')=V(\Gamma)\cup \{p_0,p_1,p_2,p_3\}$. Let $e_0, e_1, e_2, e_3, e_4$ be the edges of $\Gamma'$ over $e$ in such that $\sigma(e_i) = p_i$ and $\tau(e_i) = p_{i+1}$. In particular, we have that $\ell'(e_1)=\ell'(e_3)$. Consider the triple $f = (\Gamma', \ell', \vr)$ such that $\vr(\pm e_1) = \pm a$ and $\vr(\pm e_3) = \mp a$ and $\vr(e')=0$ for every $e'\in \ora{E}(\Gamma')\setminus\{\pm e_1, \pm e_3\}$. Notice that $f$ is a rational function on $X$, since $a\otimes \ell(e_1) + (-a)\otimes \ell(e_2) = 0$. Moreover, $\divi(f) = -ap_1 + ap_2 + ap_3 - ap_4$, and we are done
 \end{proof}

 For each stable graph $\Gamma$, we have an associated cone $\sigma_{\Gamma}:=\mathbb{R}_{\geq 0}^{E(\Gamma)}$ whose interior parametrizes genus-$g$ stable tropical curves with $n$ legs and with canonical model equal to $\Gamma$. For each specialization $\Gamma\to \Gamma'$ there is an associated face morphism $\sigma_{\Gamma'}\to \sigma_{\Gamma}$ induced by the inclusion $E(\Gamma')\subset E(\Gamma)$ (note that the specialization and face morphism can be a nontrivial automorphism). The moduli space of tropical curves is the generalized cone complex. 
 \[
M_{g,n}^{\trop}:= \lim_{\longrightarrow} \sigma_{\Gamma},
 \]
 where the colimit is taken over the category $\mathbf{G}_{g,n}$.
  We refer the reader to \cite{ACP} for the definitions of generalized cone complexes and morphisms between them.

\section{The universal poset of flows}
\label{sec:poset}

In this section we introduce the universal poset of flows, which will play a fundamental role in the rest of the paper.  This poset is obtained by enriching the graphs appearing in the poset $\mathcal G_{g,n}$ with a flow whose associated divisor is equal to a fixed divisor on the graph. 
Fist of all, we consider the case of a fixed graph. Throughout this section, $A$ will be an abelian group.

  Let $\G$ be a graph.
   Given an $A$-divisor $D$ on $\Gamma$, we define the set:
   \[
   \flowp(\G,D)=\{\vr\in C_1(\G,A); \divi(\vr)=D\}.
   \]
  
  \begin{remark}
  \label{rem:deg0}
  We have that $\flowp(\G,D)\neq \emptyset$ if and only if $D$ has degree $0$ (see Lemma \ref{lem:surjective}).
  \end{remark}
  
  \begin{remark}
  \label{rem:tree_flow_equal}
    Let $\G$ be a tree-like graph and $D$ be a divisor on $\G$. Given elements $\varphi_1,\varphi_2\in \flowp(\G,D)$, we have that  $\varphi_1(e)=\varphi_2(e)$ for every non-loop edge $e\in \ora{E}(\G)$.
  \end{remark}
   
   The following result tells us that
  the set $\flowp(\G,D)$ is well-behaved under specializations of graphs.
  
  \begin{prop}\label{prop:surjective}
    If $\iota\col \G\to \G'$ is a specialization of graphs and $D$ is an $A$-divisor on $\Gamma$, then the induced map $\iota_*\col \flowp(\G,D)\to \flowp(\G',\iota_*(D))$ is surjective. If $b_1(\G)=b_1(\G')$, then $\iota_*$ is bijective.
  \end{prop}
     \begin{proof}
       We can assume that $\iota$ is the contraction of a single edge $e_0\in \ora{E}(\G)$. Let $\vr'\in \flowp(\G',\iota_*(D))$. If $v$ is the vertex to which $e_0$ contracts, then
       \[
       \sum_{\substack{e\in \ora{E}(\G')\\ \tau(e)=v}}\vr'(e)=\iota_*(D)(v)=\begin{cases}
       D(\sigma(e_0))+D(\tau(e_0))&\text{ if $e_0$ is not a loop}\\
       D(\sigma(e_0))&\text{ if $e_0$ is a loop}.
       \end{cases}
       \]
       Define the flow $\vr$ on $\G$ by $\vr(e):=\vr'(e)$ if $e\in  E(\G)\setminus\{e_0\}=E(\G')$ (in particular, $\iota_*(\vr)=\vr'$), and
       \begin{equation}
       \label{eq:phiD}
       \vr(e_0):=-D(\sigma(e_0))+\sum_{\substack{e\in \ora{E}(\G)\setminus\{e_0\}\\ \tau(e)=\sigma(e_0)}}\vr'(e)
        =\,\,D(\tau(e_0))-\sum_{\substack{e\in \ora{E}(\G)\setminus\{e_0\}\\ \tau(e)=\tau(e_0)}}\vr'(e).
       \end{equation}
       Then we have $\divi(\vr)=D$, which finishes the proof of the first part. \par
       Moreover, if $b_1(\Gamma)=b_1(\Gamma')$, then $e_0$ is not a loop, which means that the flow $\varphi$ defined above is the unique flow satisfying $\iota_*(\varphi)=\varphi'$ and $\divi(\vr)=D$. The proof is complete.
     \end{proof}


Next, we want to cast together the set $\flowp(\G,D)$ into a universal object over $\mathcal G_{g,n}$. First we need to define a universal divisor $D$. \par

\begin{definition}\label{def:DRGamma}
 Let $g$ and $n$ be non-negative integers.  
    An \emph{$A$-ramification sequence over $\mathcal G_{g,n}$} (or, simply, \emph{ramification sequence}) is an $(n+1)$-tuple $\R=(a_1,\ldots, a_n,b)$ of elements of $A$. The \emph{degree} of an $A$ ramification sequence, denoted by $\deg\mathcal R$, is given by 
    \[
\deg\mathcal R=a_1+a_2+\ldots+a_n+(2g-2)b.\] 
    Given an $A$-ramification sequence $\R=(a_1,\ldots, a_n,b)$ and a genus-g graph with $n$ legs $\G$, we define the $A$-divisor $D_{\R,\Gamma}\in \Div^0(\G,A)$ taking  a vertex $v \in V(\G)$ to
\[
D_{\R,\Gamma}(v)=(2w_\G(v)+\val_{\Gamma}(v)-2)b+\sum_{j\in L_\G^{-1}(v)}a_j.
\]
\end{definition}

From now on in this section, we will always consider $A$-ramifications of degree zero. 
 The following remark is straightforward.
 \begin{remark}\label{rem:push-flow1}
Let $\R=(a_1,a_2,\dots,a_n,b)$ be an $A$-ramification sequence over $\mathcal G_{g,n}$.  If $\iota\col \G\to \G'$ is a specialization of genus-$g$ graphs with $n$ legs, then
$\iota_*(D_{\R,\G})=D_{\R,\G'}$.
  \end{remark}

  \begin{definition}
   \label{def:posetflowR} Let $g$ and $n$ be non-negative integers and let $\mathcal{R}$ be an $A$-ramification sequence  over $\mathcal{G}_{g,n}$.d 
We define $\flow_{g}(A,\R)$ to be the category whose objects are pairs $(\G,\varphi)$ where $\G$ is a genus-$g$ stable graph with $n$ legs and $\varphi$ is an $A$-flow on $\G$ such that $\divi(\varphi)=D_{\R,\G}$, and whose morphisms are specializations $(\G,\varphi)\to (\G',\varphi')$.
The \emph{universal poset of flows} is the poset $\flowp_{g}(A,\R)$ whose elements are the isomorphism classes in $\flow_{g}(A,\R)$ and the partial order is given by $(\G,\varphi)\geq (\G',\varphi')$ whenever there is an specialization of pairs $(\G,\varphi)\to(\G',\varphi')$.
\end{definition}



\begin{remark}
For the $A$-ramification $\R=(0,\ldots,0)$, the $A$-flows  satisfying $\divi(\varphi)=D_{\R,\Gamma}=0$ are usually called \emph{Kirchhoff flows}. In this case, we will simply write $\mathcal F_{g,n}(A)$ and $\flow_{g,n}(A)$ instead of $\mathcal F_g(A,\mathcal R)$ and $\flow_{g}(A,\R)$.
\end{remark}

\begin{theorem}\label{thm:main1}
The poset $\flowp_{g}(A,\R)$ satisfies the following properties.
\begin{enumerate}[(i)]
    \item It has a unique minimal element which is the pair $(\G,\varphi)$ where $\G$ has only one vertex with weight $g$ and no edges (and $\varphi$ is trivial). In particular, the poset $\flowp_{g}(A,\R)$ is connected.
    \item Its maximal elements  are the pairs $(\G,\varphi)$ where $\G$ is a $3$-regular stable graph with $n$-legs.
    \item It is ranked of dimension $3g-3+n$, where the rank of a pair $(\G,\varphi)$ is given by $|E(\G)|$.
    \item There is a forgetful surjective map of ranked posets $\flowp_{g}(A,\R)\to \mathcal G_{g,n}$ taking  $(\G,\varphi)$ to $\G$. When $\R=(0,\ldots, 0)$, the fiber over the class of a graph $\Gamma$ is isomorphic to $H_1(\Gamma,A)/\Aut(\Gamma)$.
\end{enumerate}
\end{theorem}

\begin{proof}
We begin by proving item (i). Clearly, every pair $(\G,\vr)$ specializes to the pair $(\G',\vr')$ where $\G'$ is the graph with one vertex of weight $g$ and no edges, and $\vr'$ is the trivial flow (just contract all edges of $\G$). This proves (i).

To prove item (ii), first recall that, given a stable genus-$g$ graph $\G$ with $n$-legs, there is a specialization $\G'\to \G$, where $\G'$ is a $3$-regular stable genus-$g$ graph with $n$-legs (see \cite[Theorem 3.2.5]{BMV}). Moreover, if $\G$ is $3$-regular, every such  specialization $\G'\to \G$ is an isomorphism. Therefore, by Proposition \ref{prop:surjective}, every pair $(\G,\vr)$ admits a specialization $(\G',\vr')\to (\G,\vr)$ where $\G'$ is $3$-regular. This proves that every maximal element of $\flowp_{g}(A,\R)$ is of the form $(\G,\vr)$ where $\G$ is $3$-regular.

To prove item (iii), we prove that $\flowp_g(A,\R)$ is ranked with rank function given by $(\G,\vr)\mapsto |E(\G)|$. To do that, it is sufficient to prove that for each specialization $(\G_1,\vr_1)\geq (\G_2,\vr_2)$ such that $|E(\G_1)|>|E(\G_2)|+1$, then there exists $(\G_0,\vr_0)$ such that $ (\G_1,\vr_1) > (\G_0,\vr_0) >(\G_2,\vr_2)$. Since $\mathcal G_{g,n}$ is ranked, we can find $\G_0$ such that $\G_1>\G_0>\G_2$, i.e., there exist specializations $\iota\col \G_1\to \G_0$ and $\G_0\to \G_2$. Defining $\vr_0:=\iota_*(\vr_1)$, then we have that $ (\G_1,\vr_1) > (\G_0,\vr_0) >(\G_2,\vr_2)$.\par

The first statement of item (iv) is clear. When $\R=(0,0,\ldots,0)$ we have that the elements of $\flowp_{g}(A)$ are pairs $(\G,\vr)$ where $\divi(\vr)=0$, modulo isomorphisms. In particular, the fiber of the forgetful map $\mathcal F_{g,n}(A)\ra \mathcal G_{g,n}$ over the class of a graph  $\G$ consists of the set 
\[
\{(\G',\vr) \colon \G'\cong \G, \vr\in H_1(\G',A)\}/\sim,
\]
where $\sim$ denotes isomorphism of pairs, and so this fiber can be identified with $H_1(\Gamma,A)/\Aut(\Gamma)$.
\end{proof}

We compare universal posets of flows when we vary the abelian group $A$.

\begin{prop}
\label{prop:rankmap}
Let $f\col A\to B$ be a homomorphism of abelian groups and $\R=(a_1,a_2,\dots,a_n,\ell)$ be an $A$-ramification sequence. Then the natural map
\begin{align*}
f_*\colon\flowp_{g}(A,\R) &\to \flowp_{g}(B,f(\R))\\
   (\G,\varphi) &\to (\G, f\circ \varphi)
\end{align*}
is a map of ranked posets over $\mathcal G_{g,n}$. If $f$ is injective (respectively, surjective), then $f_*$ is injective (respectively, surjective). Here we denote by $f(\R)$ the $B$-ramification sequence $f(\R):=(f(a_1),f(a_2),\ldots, f(a_n),f(\ell))$.
\end{prop}

\begin{proof}
  If $(\G,\varphi)\to (\G',\varphi')$ is a specialization, then there is a specialization $(\G,f\circ \varphi)\to (\G',f\circ\varphi')$. This means that the map $f_*$ preserves the partial order. Since the rank function of both  $\flowp_{g}(A,\R)$ and $\flowp_{g}(B,f(\R))$ is the number of edges of $\G$, then the map also preserves the rank. This means that  $f_*$ is a map of ranked posets. It is clear that $f_*$ commutes with the forgetful maps to $\mathcal G_{g,n}$.
  If $f$ is injective, then it is also clear that $f_*$ is injective as well.
  
  Assume that $f$ is surjective and consider a pair $(\G,\varphi_B) \in \flowp_{g}(B,f(\mathcal{R}))$. Since $f$ is surjective, there is an $A$-flow $\widetilde{\varphi}_A$ such that $\varphi_B=f\circ\widetilde{\varphi}_A$.   Since $\divi(\varphi_B)=D_{\Gamma,f(\mathcal{R})}$, we have that $\divi(\widetilde{\vr}_A)-D_{\Gamma,\mathcal{R}}$ is an element of $\Div^0(\G,\ker(f))$. By Lemma \ref{lem:surjective}, we have that there exists a flow $\vr_{\ker(f)}$ such that $\divi(\vr_{\ker(f)})=\divi(\widetilde{\vr}_A)-D_{\Gamma,\mathcal{R}}$. This means that the flow $\vr_A\coloneqq \widetilde{\vr}_A-\vr_{\ker(f)}$ is such that $\divi(\vr_A)=D_{\Gamma,\mathcal{R}}$ and $f\circ \vr_A=\vr_B$. This proves that $(\G,\vr_B)$ is the image of $(\G,\vr_A)$ and hence the map $f_*$ is surjective. 
\end{proof}

\section{Roots of $A$-divisors}

We study roots of $A$-divisors on a fixed tropical curve. 
Throughout the section, $X$ will be a tropical curve and $A$ will be  an abelian group. 
 We let $r$ be a non-zero integer such that the multiplication map  $A\xrightarrow{\cdot r}A$ is injective (which is necessary to have a well-defined division by $r$). We let $\D$ be an $A$-divisor on the tropical curve $X$ with degree in $rA$.
  We denote by $\roots^{r,\trop}(X,\D)$ the set of $r$-roots of $\D$, i.e.,
  \[
  \roots^{r,\trop}(X,\D):=\{[\D']; r\D'\text{ is equivalent to } \D\}.
  \]

Let $(\Gamma,\ell)$ be a model of $X$ whose set of vertices contains the support of $\D$. We fix an orientation $\overrightarrow{\G}$ of $\G$. 
Let $D$ be the $A$-divisor on $\G$ corresponding to $\D$.
 We let  $\overline{D}$ be the $(A/rA)$-divisor on $\G$ induced by $D$, i.e., obtained from $D$ by taking its coefficient modulo $rA$ (clearly, $\overline{D}$ has degree $0$). Given an $A$-flow $\varphi$ on $\G $, we let $\overline{\vr}$ be the $(A/rA)$-flow on $\G$ induced by $\vr$.
 
 \begin{definition}\label{def:divisor_flow}
 For each $A$-flow $\vr$ on $\G$ such that $\overline{\vr}\in \flowp(\G,\overline{D})$, we define the divisor $\D_{\vr}$ (sometimes denoted  $\D_{\vr,r}$, or $\D_{\vr,r,D}$, when it is important to specify the dependence on $r$ and $D$) on the tropical curve $X$ as follows. For every vertex  $v\in V(\G)$, we set
  \[
 \D_\vr(v)=\frac{\D(v)-\sum_{e\in E(\ora{\Gamma}), \sigma(e)=v}(r-1)\vr(e)-\sum_{e\in E(\ora{\Gamma}),\tau(e)=v}\vr(e)}{r}  .
 \]
 Notice that the above expression is well-defined since the numerator of the above expression is in $rA$ 
 and the multiplication by $r$ is injective.
 Given $e\in E(\overrightarrow{\G})$, we let $q_{e}$ be the point on $e$ at distance $\frac{\ell(e)}{r}$ from $\sigma(e)$. We set $\D_\vr(q_{e})=\vr(e)$ for every $e\in E(\ora{\G})$. Finally, we set $\D_{\vr}(p)=0$ for every  point $p$ on $X\setminus V(\G)$ distinct from every point $q_e$. We call $\D_\vr$ \emph{the divisor associated to the flow $\vr$}.
 \end{definition}

\begin{definition}\label{def:flow_divisor}
    For each $A$-divisor $\D'$ on $X$ such that $r\D'$ is equivalent to $\D$, we define the $(A/rA)$-flow $\overline{\varphi}_{\D'}$ on $\G$ as follows. Choose  a rational function $f=(\Gamma_f, \ell_f, \vr_f)$ on $X$ such that $\D-r\D'=\divi(f)$, where $\Gamma_f$ is a refinement of $\Gamma$.    
    Since $\D$ is supported on $V(\Gamma)$, for each $e\in E(\Gamma)$ and for each point $p \in e^\circ$, we have
    \[
    \divi(f)(p)=\D(p)-r\D'(p)=-r\D'(p) \in rA.
    \]   
    Thus 
    for each $e \in \ora{E}(\G)$, and  $e',e''\in \ora{E}(\Gamma_f)$ over $e$ we have that $\vr_f(e')\equiv \vr_f(e'')\bmod rA$.
    We set $\overline{\varphi}_{\D'}(e):= \vr_f(e')\bmod rA$ for any $e'\in \ora{E}(\Gamma_f)$ over $e$. 
  \end{definition}

  In the remainder of this section, 
  we will prove that the function
 \begin{align}\label{eq:Delta}
 \Delta := \Delta_{\Gamma} \colon \flowp(\Gamma, \overline{D})& \to \roots^{r,\trop}(X,\D)\\
   \overline{\vr} &\mapsto [\D_{\vr}]\nonumber
 \end{align}
 where $\vr\in \mathcal F(\G,D)$ is a lifting of a flow $\overline{\vr}\in\flowp(\Gamma, \overline{D})$,
 and the function
 \begin{align}\label{eq:Phi}
     \Phi = \Phi_{\Gamma} \colon \roots^{r, \trop}(X,\D) &\to  \flowp(\Gamma,\overline{D})\\
            \D'& \mapsto \overline{\vr}_{\D'}\nonumber
 \end{align}
 are well-defined and are one inverse to each other. We begin by proving some auxiliary results.

 \begin{lemma}
 \label{lem:segment}
 Let $e\in \ora{E}(\G)$.  Denote by
 $q_e$ and $q_{-e}$ the points at distance $\frac{\ell(e)}{r}$ from $\sigma(e)$ and $\tau(e)$, respectively.
 For every $a \in A$, there are rational functions $f$ and $f'$ on $X$ such that
 \begin{enumerate}[(i)]
     \item $\divi(f)=-(r-1)a\sigma(e)+raq_e-a\tau(e)$.
     \item $f'(\sigma(e))=f'(\tau(e))$ and $\divi(f')=-a\sigma(e)+aq_e+aq_{-e}-a\tau(e)$.
 \end{enumerate} 
 \end{lemma}
 
 \begin{proof}     
    We prove (i). Consider the refinement $\Gamma_f$ of $\Gamma$ obtained by inserting two vertices $v_e$ and $v_{-e}$ over $e$. We let $e_0,e_1,e_2\in \ora{E}(\Gamma)$ be the edges over $e$ such that $\sigma(e_0)=\sigma(e)$, $\tau(e_0)=\sigma(e_1)=v_e$, $\tau(e_1)=\sigma(e_2)=v_{-e}$ and $\tau(e_2)=\tau(e)$. We let $\ell_f$ be the length function on $\Gamma_f$ which agrees with $\ell$ on every edge $e'\in  E(\Gamma)\setminus\{e\}$ and $\ell_f(e_0)=\ell_f(e_2)=\frac{\ell(e)}{r}$ and $\ell_f(e_1)=\frac{\ell(e)(r-2)}{r}$. In this way $(\Gamma_f,\ell_f)$ is a model of $X$ and $v_{e}$ and $v_{-e}$ correspond to the points $q_e$ and $q_{-e}$ on $X$. Consider the rational function $f=(\Gamma_f, \ell_f,\vr_f)$ on $X$, where $\vr_f$ is the flow on $\Gamma_f$ taking an element $e'\in \ora{E}(\Gamma_f)$ to 
    \[
    \vr_f(e')=\begin{cases}
          \pm (r-1)a& \text{ if }e'=\pm e_0\\
          \mp a& \text{ if }e'=\pm e_1,\pm e_2\\
          0& \text{ otherwise.}
    \end{cases}
    \]
    Then $f$ is a rational function on $X$ (see Equation \eqref{eq:rationalfunction}), because $(r-1)a\otimes \frac{\ell(e)}{r}+(-a)\otimes \frac{\ell(e)(r-1)}{r}=0$. Moreover, we have that $\divi(f)=-(r-1)a\sigma(e)+raq_e-a\tau(e)$. This concludes the proof of (i).
    
Notice that (ii) is a particular case of Lemma \ref{lem:principal}.
 \end{proof}
 
 In the next three results we will prove that the map $\Delta$ in Equation \eqref{eq:Delta} is well-defined.
 We begin by proving that $\Delta$  is independent of the choice of orientation $\ora{\Gamma}$ of $\Gamma$.
  \begin{prop}
  \label{prop:Delta_orientation}
 The class $[\D_{\vr}]$ 
 does not depend on the choice  of orientation of $\Gamma$.
 \end{prop}
 \begin{proof}
 It is enough to prove the result when two orientations  differ on just one edge $e_0\in \ora{E}(\Gamma)$. In this case, if $\D_{\vr}$ and $\D_{\vr}'$ are the divisors associated to the flow $\vr$ constructed using the two orientations, then 
 \begin{equation}\label{eq:difference}
 \D_{\vr}-\D'_{\vr}=-\vr(e_0)\sigma(e_0)+\vr(e_0)q_{e_0}+\vr(e_0)q_{-e_0}-\vr(e_0)\tau(e_0).
 \end{equation}
 The divisor in Equation \eqref{eq:difference} is principal, as we can see by taking $a = \vr(e_0)$ in Lemma \ref{lem:segment} (ii).
 \end{proof}

 \begin{prop}
 \label{prop:Delta_lifting}
 The class $[\D_{\vr}]$ 
 does not depend on the choice of the lifting $\vr\in C_1(\G,A)$ of $\ol{\vr}$.
 \end{prop}
 \begin{proof}
    Consider an edge $e_0\in \ora{E}(\G)$.
    It is enough to prove that, if $\vr'$ is a flow on $\G$ such that $\vr'( e_0)=\vr( e_0)+ ra$ for some $a\in A$, and $\vr'(e)=\vr(e)$ for every $e\in \ora{E}(\Gamma)\setminus\{\pm e_0\}$, then $\D_{\vr'}$ is equivalent to $\D_{\vr}$.
    Indeed, we have $\D_{\vr'}-\D_{\vr}=-(r-1)a\sigma(e_0)+raq_{e_0}-a\tau(e_0)$. By Lemma \ref{lem:segment} (i) there exists a rational function $f$ on $X$ such that $\divi(f)=-(r-1)a\sigma(e_0)+raq_{e_0}-a\tau(e_0)$. This means that $\D_{\vr'}-\D_{\vr}$ is principal on $X$, and we are done.
 \end{proof}

    \begin{Cons}\label{cons:GammaS}
      Let $X$ be a tropical curve with model $(\Gamma,\ell)$. For a set of points $S \subset X\setminus V(\Gamma)$ we define the model $(\Gamma_S,\ell_S)$ of $X$ such that $V(\Gamma_S)=V(\Gamma)\cup S$. Notice that $\Gamma_S$ is a refinement of $\Gamma$. For every orientation  $\ora{\Gamma}$ on $\Gamma$, there is a natural induced orientation $\ora{\Gamma_S}$ on $\G_S$. In this case, given $e\in E(\ora{\G})$, we denote by $e^{(j)}_S$, for $j=0,\ldots, m_e$, the oriented edges in $E(\ora{\Gamma_S})$ over $e$ such that $\sigma(e^{(0)}_S)=\sigma(e)$, $\tau(e^{(m_e)}_S)=\tau(e)$ and $\sigma(e^{(i+1)}_S)=\tau(e^{(i)}_S)$. We call $(\Gamma_S,\ell_S)$ \emph{the model of $X$ induced by} $S$. Notice that if $\vr$ is a flow on $\Gamma$, then we have an induced flow $\vr_S$ on $\Gamma_S$ such that $\vr_S(e^{(j)}_S)=\vr(e)$ for every $e\in E(\overrightarrow{\Gamma})$ and  $j=0,\dots,m_e$.
    \end{Cons}
 
 \begin{corollary}
 \label{cor:flow_bijection}
  Let $D$ be a divisor on $\Gamma$. Consider a set $S\subset X\setminus V(\Gamma)$. Let $D_S$ be the induced divisor on $\Gamma_S$ and $\overline{D}_S$ the associated $A/rA$-divisor. Then the function \begin{align*}
      \flowp(\Gamma, \overline{D})& \to \flowp(\Gamma_S, \overline{D}_S)\\
       \vr & \mapsto \vr_S
  \end{align*}
 is a bijection whose inverse is $\iota_*\col \flowp(\Gamma_S, D_S)\to \flowp(\Gamma,D)$ where $\iota$ is any specialization $\iota\col \Gamma_S\to \Gamma$ such that $\iota^*(e)$ is an edge over $e$ for every $e\in E(\Gamma)$.
\end{corollary}
\begin{proof}
   The result follows directly from Construction \ref{cons:GammaS} and Proposition \ref{prop:surjective}.
\end{proof}

 \begin{prop}\label{prop:Delta_refinement}
   Consider a set of points $S\subset X\setminus V(\Gamma)$. Let $\vr$ be a flow on $\Gamma$ and $\vr_S$ be the induced flow on $\Gamma_S$. Then $[\D_\vr]=[\D_{\vr_S}]$.
 \end{prop}
 \begin{proof}
    It is sufficient to prove the proposition when $S$ is a single point $p_0$ in the interior $e^\circ$ of some edge $e\in E(\Gamma)$. Consider the points $q_e$, $q_{e^{(0)}_S}$, $q_{e^{(1)}_S}$ of $X$ (recall the notation in Definition \ref{def:divisor_flow}).
    In this case, we have that
        \[
    \D_{\vr_S}-\D_{\vr} = q_{e^{(0)}_S} - q_e - \sigma(e^{(1)}_S) + q_{e^{(1)}_S}.
        \]
        The distance between $q_{e^{(0)}}$ and $q_e$ is equal to $\frac{\ell(e)-\ell(e^{(0)})}{r}$ and the distance between $\sigma(e^{(1)})$ and $q_{e^{(1)}}$ is equal to $\frac{\ell(e^{(1)}_S)}{r} = \frac{\ell(e)-\ell(e^{(0)}_S)}{r}$.    The result follows from Lemma \ref{lem:principal}. Notice that it can happen that $\sigma(e^{(1)}_S)$ is between $q_{e^{(0)}_S}$ and $q_e$: in this case, the distance between $q_{e^{(0)}_S}$ and $\sigma(e^{(1)}_S)$ is the same as the distance between $q_e$ and $q_{e^{(1)}_S}$.
 \end{proof}

 \begin{prop}
 \label{prop:Delta_root}
 Consider a flow $\overline{\vr}\in \flowp(\Gamma,\overline{D})$. For every lifting $\vr\in C_1(\G,A)$ of $\overline{\vr}$, we have that $r\D_{\vr}$ is equivalent to $\D$.
 \end{prop}
 
\begin{proof}
For every $e\in E(\ora{\G})$, we set $\D_e:=\vr(e)(-(r-1)\sigma(e)+rq_e-\tau(e))$.
By Definition \ref{def:divisor_flow}, we have
\begin{equation}
\label{eq:DphiDe}
r\D_{\vr}=\D+\sum_{e\in E(\ora{\G})} \D_e.    
\end{equation}
By Lemma \ref{lem:segment} (i), the divisor  $\D_e$ is principal on $X$, hence the result follows.
\end{proof}
\begin{corollary}
\label{cor:Delta_well}
 The map $\Delta$ is well-defined.
\end{corollary}

\begin{proof}
   Just combine Propositions \ref{prop:Delta_orientation}, \ref{prop:Delta_lifting}, and \ref{prop:Delta_root}.
\end{proof}

Recall that the definition of the map $\Delta=\Delta_{\Gamma}$ in Equation \eqref{eq:Delta} is given was given after fixing a model $\Gamma$ of $X$. 

\begin{corollary}
\label{cor:Delta_refinement}
Consider a set $S\subset X\setminus V(\Gamma)$. We have a natural commutative diagram:
 \[
 \begin{tikzcd}[row sep = tiny]
  \flowp(\Gamma, \overline{D}) \arrow[rd, "\Delta_\Gamma"] \arrow[dd, leftrightarrow,"\iota_*"] & \\
  &  \Roots^{r,\trop}(X, \mathcal{D}) \\
\flowp(\Gamma_S, \overline{D}_S) \arrow[ru, "\Delta_{\Gamma_S}"']
 \end{tikzcd}
 \]
 where $\iota_*$ is the bijection induced by  any specialization $\iota\col \Gamma_S\to \Gamma$ as in Corollary \ref{cor:flow_bijection}.
\end{corollary}
\begin{proof}
   The result follows from Corollary  \ref{cor:flow_bijection} and Proposition \ref{prop:Delta_refinement} .
\end{proof}

The next result is a fundamental step toward the proof that the maps $\Delta$ and $\Phi$ are inverse to each other.

\begin{prop}
\label{prop:Delta_injective}
 The map $\Delta$ is injective.
\end{prop}
\begin{proof}
Let $\overline{\vr}_1$ and $\overline{\vr}_2$ be flows in  $\flowp(\G,\overline{D})$. Consider liftings $\vr_1$ and $\vr_2$ in $C_1(\G,A)$ of $\overline{\vr}_1$ and $\overline{\vr}_2$, respectively. We need to prove that if $\D_{\vr_1}$ is equivalent to $\D_{\vr_2}$, then $\overline{\vr}_1=\overline{\vr}_2$.\par

   Define $\vr:=\vr_1-\vr_2$, and let $\overline{\vr}$ be the $(A/rA)$-flow on $\G$ induced by $\vr$. We set  $\D_{\vr,0}:=\D_{\vr_1}-\D_{\vr_2}$. Notice that
   \[
    \D_{\vr,0}(v)=\D_{\vr_1}(v)-\D_{\vr_2}(v)=\frac{-\sum_{e\in E(\ora{\Gamma}), \sigma(e)=v}(r-1)\vr(e)-\sum_{e\in E(\ora{\Gamma}),\tau(e)=v}\vr(e)}{r}  
 \]
for every $v\in V(\Gamma)$. Moreover,
  $\D_{\vr, 0}(q_{e})=\vr(e)$ for every $e\in E(\ora{\Gamma})$, and $\D_{\vr,0}(p) = 0$ for every other point $p\in X$. 
   We deduce that  $\D_{\vr_1}$ is equivalent to $\D_{\vr_2}$ if, and only if, $\D_{\vr, 0}$ is principal. Notice that  $\divi(\overline\vr)=\divi(\overline{\vr}_1)-\divi(\overline{\vr}_2)=\ol{0}$ (this is the trivial $(A/rA)$-divisor on $\G$), hence in particular $\overline{\vr}\in \flowp(\G,\ol{0})$.\par
    Let us assume that $\D_{\vr, 0}$ is principal. We need to prove that $\overline{\vr}$ is the trivial flow or, equivalently, $\vr(e)\in rA$ for every $e\in E(\ora{\Gamma})$.
     Consider the set $S = \{q_e\}_{e\in E(\ora{\Gamma})}$ and the model $(\Gamma_S,\ell_S)$ of $X$ induced by $S$. Since $\D_{\vr,0}$ is principal and supported on $V(\Gamma_S)$, there exists a rational function $f=(\Gamma_S,\ell_S,\vr_f)$ on $X$ such that $\divi(f)=\D_{\vr,0}$.\par 
      By the condition $\divi(f)=\D_{\vr,0}$, we deduce that \begin{equation}\label{eq:fieffe}
          \sum_{\substack{e\in E(\ora{\Gamma})\\\tau(e)=v}}\vr_f(e^{(1)}_S)-\sum_{\substack{e\in E(\ora{\Gamma})\\\sigma(e)=v}}\vr_{f}(e^{(0)}_S)=\frac{-\sum_{e\in E(\ora{\Gamma}), \tau(e)=v}\vr(e)-\sum_{e\in E(\ora{\Gamma}),\sigma(e)=v}(r-1)\vr(e)}{r}
      \end{equation}
      for every $v\in V(\Gamma)$ and
      \begin{equation}
          \vr_f(e^{(0)}_S)-\vr_f(e^{(1)}_S)= \D_{\vr, 0}(q_e) = \vr(e)   \label{eq:phifphi}
          \end{equation}
          for every $e\in E(\ora{\Gamma})$.
    Therefore, multiplying Equation \ref{eq:fieffe} by $r$ and using Equation \ref{eq:phifphi}, we get
    \[
    r\sum_{\substack{e\in E(\ora{\Gamma})\\\tau(e)=v}}\vr_f(e^{(1)}_S)-r\sum_{\substack{e\in E(\ora{\Gamma})\\\sigma(e)=v}}\vr_{f}(e^{(0)}_S)=-\sum_{\substack{e\in E(\ora{\Gamma})\\\tau(e)=v}}(\vr_f(e^{(0)}_S)-\vr_f(e^{(1)}_S))-\sum_{\substack{e\in E(\ora{\Gamma})\\\sigma(e)=v}}(r-1)(\vr_f(e^{(0)}_S)-\vr_{f}(e^{(1)}_S)),
    \]
    and hence
    \begin{equation}
        \label{eq:beta}
    -\sum_{\substack{e\in E(\ora{\Gamma})\\\sigma(e)=v}}(\vr_f(e^{(0)}_S)+(r-1)\vr_{f}(e^{(1)}_S))+\sum_{\substack{e\in E(\ora{\Gamma})\\\tau(e)=v}}(\vr_f(e^{(0)}_S)+(r-1)\vr_f(e^{(1)}_S))=0.
    \end{equation}
    
    Let $\beta$ be the $A$-flow on $\Gamma$ defined by $\beta(e)=\vr_f(e^{(0)}_S)+(r-1)\vr_f(e^{(1)}_S)$ for every edge $e\in E(\ora{\G})$. By Equation \eqref{eq:beta}, we have that $\divi(\beta)=0$.  Let $\overline{\beta}$ be the $(A/rA)$-flow on $\G$ induced by $\beta$. By Equation \eqref{eq:phifphi}, we have
    \[
    \beta(e)=\vr_f(e^{(0)}_S)-\vr_f(e^{(1)}_S)+r\vr_f(e^{(1)}_S)=\vr(e)+r\vr_f(e^{(1)}_S),
    \]
    and hence $\overline{\beta}=\overline{\vr}$ in $C_1(\Gamma,A/rA)$.
    
    Since $f$ is a rational function on $X$ (recall Equation \eqref{eq:rationalfunction-gamma}), we have
    \[
    \sum_{e\in \gamma}\gamma(e)\bigg(\vr_f(e^{(0)}_S)\otimes\frac{\ell(e)}{r}+\vr_f(e^{(1)}_S)\otimes\frac{(r-1)\ell(e)}{r}\bigg)=0,
    \]
    for every cycle $\gamma$ of $\ora{\Gamma}$. In particular,  we get
    \[
    0=\sum_{e\in \gamma}\gamma(e)(\vr_f(e^{(0)}_S)+(r-1)\vr_f(e^{(1)}_S))\otimes \ell(e)=\sum_{e\in \gamma}\gamma(e)(\beta(e)\otimes \ell(e))
    \]
    for every cycle $\gamma$ of $\ora{\Gamma}$.
    Therefore, $f_\beta:=(\Gamma,\ell,\beta)$ is a rational function on $X$ such that $\divi(f_{\beta})=0$. By Lemma \ref{lem:rational-function-torsion}, we have that $\beta$ is a torsion flow, hence, by Lemma \ref{lem:torsion_quotient}, we have that $\overline{\vr}=\overline{\beta}=0$. This finishes the proof.
    \end{proof}
    
    Next we turn our attention to the map $\Phi$ in Equation \eqref{eq:Phi}. First we prove that $\Phi$ is well-defined. Recall the notation in Definition \ref{def:flow_divisor}.

\begin{prop}
\label{prop:Phi_well}
   Let $\D'$ be a divisor on $X$ such that $[\D']\in \roots^{r,\trop}(X,\D)$. Then the following properties hold.
   \begin{enumerate}[(i)]
       \item The flow $\overline{\varphi}_{\D'}$ does not depend on the choice of the rational function $f=(\Gamma_f,\ell_f,\vr_f)$.
       \item If $\D''$ is a divisor on $X$ such that $[\D'']=[\D']$ then $\overline{\vr}_{\D'}=\overline{\vr}_{\D''}$.
       \item We have $\divi(\overline{\vr}_{\D'})\in \flowp(\Gamma, \overline{D})$.
   \end{enumerate}
\end{prop}

\begin{proof}
       Let $h=(\Gamma_h,\ell_h,\vr_h)$ be another rational function on $X$ such that $\divi(h)=\divi(f)$. Upon taking a common refinement of $\G$, we can assume that $\Gamma_h=\Gamma_f$ and $\ell_f=\ell_h$. By Lemma \ref{lem:rational-function-torsion}, we have that $\vr_f-\vr_h$ is a torsion flow and hence, by Lemma \ref{lem:torsion_quotient}, this flow vanishes modulo $rA$. This finishes the proof of (i).
      
      Let $f=(\Gamma_f,\ell_f,\vr_f)$ and $h=(\Gamma_h,\ell_h,\vr_h)$ be rational functions on $X$ such that $\Gamma_f=\Gamma_h$, $\ell_f=\ell_h$, $\divi(f)=\D-r\D'$ and $\divi(h)=\D'-\D''$. It follows that $\divi(f+rh)=\D-r\D''$. Since $\vr_{f}$ and $\vr_{f+rh}=\vr_{f}+r\vr_{h}$ are equal modulo $rA$, this ends the proof of (ii).
           
     Let $f$ be a rational function on $X$ such that  $\D-r\D'=\divi(f)$. Hence $\overline{\D}=\overline{\divi(f)}$, which means that $\overline{D}=\overline{\divi(\vr_f)}$. Since $\overline{\divi(\varphi_f)}=\divi(\overline{\vr}_{\D'})$, we are done.
\end{proof}

\begin{corollary}
\label{cor:Phi_well}
 The map $\Phi$ is well-defined.
\end{corollary}

    \begin{prop}
    \label{prop:Delta_Phi_identity}
    The composition $\Delta\circ \Phi$ is the identity.
    \end{prop}
    \begin{proof}
       Let $[\D']$ be a class in $\roots^{r,\trop}(X,D)$. To prove the result, let  $\vr$ be a lifting of the flow $\overline{\vr}_{\D'}=\Phi([\D'])$. We will prove that the divisor $\D_\vr$ is equivalent to $\D'$. Note that $[\D_{\vr}]=\Delta(\overline{\vr}_{\D'})=\Delta\circ\Phi([\D'])$.\par

       By Corollary \ref{cor:Delta_refinement} we can assume that the support of $\D'$ is contained in $V(\Gamma)$. We let $f=(\Gamma,\ell,\vr_f)$ be a rational function on $X$ such that  $\D-r\D'=\divi(f)$. By Definition \ref{def:flow_divisor}, we have that $\vr_f$ is a lifting of $\overline{\vr}_{\D'}$. By Proposition \ref{prop:Delta_lifting}, we can assume that $\vr=\vr_f$.\par 
       
       Consider the set $S= \{q_e\}_{e\in E(\ora{\Gamma})}$ (see Definition \ref{def:divisor_flow}) and define the flow $\beta$ on $\Gamma_S$ such that $\beta(e^{(0)}_S)=\vr(e)$ and $\beta(e^{(1)}_S)=0$ for every $e\in E(\ora{\Gamma})$. We will prove that $h=(\Gamma_S, \ell_S, \beta)$ is a rational function on $X$ such that $\D_{\vr}-\D'=\divi(h)$.

       We begin by proving that $h$ is a rational function. Since $f$ is a rational function, we have 
       \begin{align*}
       0 & =\sum_{e\in E(\ora{\Gamma)}}\gamma(e)\left(\beta(e^{(0)}_S)\otimes \frac{\ell(e)}{r}+\beta(e^{(1)}_S)\otimes\frac{\ell(e)}{r}\right) 
         =\sum_{e\in E(\ora{\Gamma})}\gamma(e)\left(\vr(e)\otimes \frac{\ell(e)}{r}\right)
       \end{align*}
       for every cycle $\gamma$. Hence $h$ is a rational function by Equation \eqref{eq:rationalfunction-gamma}.
       
       Now, let us prove that $\D_{\vr}-\D' = \divi(h)$. Indeed, for every vertex $v\in V(\Gamma)$, we have that
       \begin{align*}
       \D_{\vr}(v)-\D'(v) &= \frac{\D(v)-\sum_{\sigma(e)=v}(r-1)\vr(e)-\sum_{\tau(e)=v}\vr(e)}{r}-\D'(v)\\
                          & =\frac{(\D(v) -r\D'(v))-\sum_{\sigma(e)=v}(r-1)\vr(e)-\sum_{\tau(e)=v}\vr(e)}{r} \\
                          & = \frac{\divi(f)-\sum_{\sigma(e)=v}(r-1)\vr(e)-\sum_{\tau(e)=v}\vr(e)}{r}\\
                          & = \frac{-\sum_{\sigma(e)=v}\vr(e)+\sum_{\tau(e)=v}\vr(e)-\sum_{\sigma(e)=v}(r-1)\vr(e)-\sum_{\tau(e)=v}\vr(e)}{r}\\
                          & = -\sum_{\sigma(e)=v}\vr(e) = \divi(h)(v),
       \end{align*}
       where all the sums are over $e\in E(\ora{\Gamma})$. 
       Moreover, for every $e\in E(\ora{\Gamma})$ we have 
       \[
       \D_{\vr}(q_e)-\D'(q_e)=\vr(e) = \divi(h)(q_e).
       \]
       For every other point $p\in X\setminus (V(\Gamma)\cup S)$, we have $\D_{\vr}(p)=\D'(p)=\divi(h)(p)=0$. We deduce that $\D_{\vr}-\D'=\divi(h)$, and the proof is complete.
        \end{proof}

    We sum up the results obtained so far in the following theorem.
    
    \begin{theorem}
    \label{thm:Delta_Phi}
       The maps $\Delta$ and $\Phi$ in Equations \eqref{eq:Delta} and \eqref{eq:Phi} are well-defined, they are bijections and are inverse to one another.
    \end{theorem}
    \begin{proof}
 By Corollaries \ref{cor:Delta_well} and \ref{cor:Phi_well}, the functions $\Delta$ and $\Phi$ are well-defined. The function $\Delta $ is injective by Proposition \ref{prop:Delta_injective} and $\Delta\circ \Phi$ is the identity by Proposition \ref{prop:Delta_Phi_identity}. Hence $\Delta$ and $\Phi$ are bijections and they are inverse to one another. 
    \end{proof}
    
\section{Moduli of roots of $A$-divisors}

In this section we use the universal poset of flows to construct a moduli space parametrizing roots of an $A$-divisor over a $n$-pointed tropical curve of genus $g$.

Next, we consider the universal setting. We construct a moduli space parametrizing roots of $A$-divisors on tropical curves. Throughout the section, we let $A$ be an abelian group and $r$ a non-zero integer such that the multiplication map $A\xrightarrow{\cdot r}A$ is injective. We let $\R=(a_1,\dots,a_n,b)$ be an $A$-ramification sequence and $g$ be a non-negative integer such that $a_1+\ldots+a_n+(2g-2)b\in rA$. We denote by $\ol{\R}$ the sequence whose entries are the entries of $\mathcal{R}$ taken modulo $rA$.  
 For a genus-$g$ stable tropical curve $X$ with $n$ legs having canonical model $(\Gamma,\ell)$, we define the divisor $\D_{\R,X}$ on $X$ to be the divisor induced by the divisor $D_{\R,\Gamma}$ on $\G$ (recall Definition \ref{def:DRGamma}). 
 
The first step is the definition of a generalized cone complex. 
For each pair $(\G,\ol{\vr})$ in the category $\flow_g(A/rA,\overline{\R})$, we define the cones
\[
\kappa_{(\G,\ol{\vr})}^{\circ} :=\mathbb{R}_{>0}^{E(\G)}\text{ and }\kappa_{(\G,\ol{\vr})} :=\mathbb{R}_{\geq0}^{E(\G)}.
\]
The vector space $\mathbb R^{E(\G)}$ has a canonical base indexed by the set $E(\G)$. 
In what follows we will denote by $x_e$ the coordinate corresponding to an edge $e\in E(\G)$.

Notice that each specialization $(\G,\ol{\vr})\to (\G',\ol{\vr}')$ induces a face morphism $\kappa_{(\G',\ol{\vr}')}\to\kappa_{(\G,\ol{\vr})}$.

\begin{definition}\label{def:moduli-root}
    We define the generalized cone complex $\Roots^{r, \trop}_g(A,\R)$ to be the colimit
\[
\Roots^{r,\trop}_g(A,\R):=\lim_{\longrightarrow} \kappa_{(\G,\ol{\vr})}=\bigsqcup\; \kappa_{(\G,\ol{\vr})}^{\circ}/\Aut(\Gamma,\ol{\vr}),
\]
where the colimit is taken over the category $\flow_g(A/rA,\overline{\mathcal R})$ and the union is taken over the poset $\flowp_g(A/rA,\overline{\mathcal R})$.
We define the extended generalized cone complex $\overline{\Roots}^{r, \trop}_g(A,\R)$ as the compactification of $\Roots^{r, \trop}_g(A,\R)$, that is
\[
\overline{\Roots}^{r,\trop}_g(A,\R):=\lim_{\longrightarrow} \overline{\kappa}_{(\G,\ol{\vr})}
\]
where $
\overline{\kappa}_{(\G,\ol{\vr})} :=\overline{\mathbb{R}}_{\geq 0}^{E(\G)}$ and $\overline{\mathbb{R}}:=\mathbb{R}\cup\{\infty\}$, and the colimit is taken over the category $\flow_g(A/rA,\overline{\mathcal R})$. When $A=\mathbb{Z}$ we will simply write $\Roots^{r,\trop}_g(\R)$ and $\overline{\Roots}^{r,\trop}_g(\R)$ instead of $\Roots^{r,\trop}_g(\mathbb{Z},\R)$ and $\overline{\Roots}^{r,\trop}_g(\mathbb{Z},\R)$.
  \end{definition}

There is a natural map 
\[
\pi_{g,\R}^{r, \trop}\col \Roots^{r,\trop}_g(A,\R)\to M^{\trop}_{g,n}
\]
induced by the map of cones
taking a point in the cone $\kappa_{(\G,\ol{\varphi})}$ with coordinates $(x_e)_{e\in E(\G)}$  to the point $(y_e)_{e\in E(\G)}$ of the cone $\sigma_\G=\mathbb R^{E(\G)}_{\ge0}$, where\footnote{
As we shall see later, the multiplying factor in the equation above is chosen to make the tropicalization map compatible with the geometric map. However, depending on the choice of compactification of the moduli of roots, this factor can be rescaled.} 
\begin{equation}\label{eq:ye}
y_e=\left|\frac{r}{\gcd(r,\ol{\vr}(e))}\right|\cdot x_e.
    \end{equation}
(Recall that we have natural maps $\sigma_\G\to M_{g,n}^{\trop}$.)

\begin{theorem}\label{thm:main-moduliroot}
   The generalized cone complex  $\Roots^{r,\trop}_g(A,\R)$ satisfies the following properties.
   \begin{enumerate}[(i)]
       \item The points of $\Roots^{r,\trop}_g(A,\R)$ parametrize pairs $(X,[\D])$, where $X$ is a genus-$g$ stable tropical curve with $n$ legs and $\D$ is an $r$-root of $\D_{\R,X}$.
       \item The map $\pi_{g,\R}^{r, \trop}$ is a map of generalized cone complexes  whose fiber over the class of a tropical curve $X$ is isomorphic to $\Roots^{r,\trop}(X,\D_{\R,X})/\Aut(X)$.
   \end{enumerate}
 
\end{theorem}
\begin{proof}
  Given a point in $\kappa_{\Gamma,\ol{\vr}}$ with coordinates $(x_e)_{e\in E(\G)}$, we construct the pair $(X,[\Delta(\ol{\vr})])$, where $X=(\Gamma,\ell)$ is the tropical curve with associated function $\ell\col E(\G)\to \mathbb R_{>0}$ taking $e\in E(\G)$ to
\begin{equation}
\ell(e)=\left|\frac{r}{\gcd(r,\overline{\vr}(e))}\right|\cdot x_e,
\end{equation}
 and where $[\Delta(\ol{\vr})]$ is class of the divisor $\Delta(\ol{\vr}))\in \roots^{r,\trop}(X,\D_{\R,X})$ (recall Equation \eqref{eq:Delta}). A standard argument (see \cite[Proposition 5.12]{APPLMS}) shows that points in $\kappa_{(\Gamma,\ol{\vr})}^\circ$ that are identified via $\Aut(\Gamma,\ol{\vr})$ give rise to isomorphic pairs.

  
  Vice-versa, let $X$ be a stable $n$-pointed genus-$g$ tropical curve with model $(\Gamma_X,\ell_X)$, and $\D$ be a $r$-root of $\D_{X,\R}$. By Theorem \ref{thm:Delta_Phi}, there exists a unique flow $\ol{\vr} \in \F(\Gamma_X, \overline{D}_{\R,\Gamma_X})$ such that $[\D] = \Delta(\overline{\vr})$. Hence the pair $(X,[\D])=(X,\Delta(\ol{\vr}))$ corresponds to the point $(x_e)=(\ell_X(e))_{e\in E(\G_X)}$ in $\kappa_{(\Gamma,\overline{\vr})}^\circ/\Aut(\Gamma,\ol{\vr})$. Moreover, if $(X',[\D'])$ is another pair such that $X$ and $X'$ are isomorphic tropical curve and this isomorphism takes $[\D]$ to $[\D']$, then the above construction takes $(X,[\D])$ and $(X',[\D'])$ to the same point in $\kappa_{(\Gamma,\overline{\vr})}^\circ/\Aut(\Gamma,\ol{\vr})$. This finishes the proof of (i).\par
   
   
   The map $\pi_{g,\R}^{r,\trop}$ is already defined as a map of generalized cone complexes. Let $X$ be a stable $n$-pointed genus-$g$ tropical curve and let $(\Gamma_X,\ell_X)$ be its canonical model. Consider the map \[
   h\col \Roots^{r,\trop}(X,\D_{\R, X})\to \Roots_g^{r,\trop}(A,\R)
   \]
   given by $h([\D])=[(X,[\D])]$. We have that $(\pi_{g,\R}^{r, \trop})^{-1}([X])=\text{Im}(h)$. Moreover, the group $\Aut(X)$ acts on  $\Roots^{r,\trop}(X,\D_{\R,X})$. Two $r$-roots  $[\D]$ and $[\D']$ in $\Roots^{r,\trop}(X,\D_{\R,X})$ have the same image via $h$ if and only if there is an automorphism $\alpha$ of $X$ such that $\alpha_*([\D])=[\D']$. This finishes the proof of (ii).
   \end{proof}

\begin{prop}\label{prop:homeoRprime}
Consider integers $r'$ and $d$ such that $r=dr'$ for some integer $d$. 
Let $\R'=(a_1,\ldots, a_n, \ell)$ be a ramification sequence such that $a_1+\ldots+a_n+(2g-2)\ell\in r'A$ and let $\R=(da_1,\ldots, da_n, d\ell)$. Then
there is a natural map of generalized cone complexes
\[
\Roots^{r',\trop}_g(A,\R')\to \Roots^{r,\trop}_g(A,\R)
\]
over $M_{g,n}^{\trop}$,
taking the class of $(X,[\D])$ to itself. This map is a homeomorphism onto its image.
\end{prop}

\begin{proof}
First notice that the multiplication map by $d$ is injective in the abelian group $A$, since so is the multiplication by $r$.
 We have a natural map
 \begin{align*}
     A/r'A&\to A/rA\\
       a+r'A &\mapsto da+rA.
 \end{align*}
  Then there exists a natural map of multiplication by $d$ given by
 \begin{align*}
  \flowp_{g}(A/r'A,\R'+r'A) &\to \flowp_g(A/rA, \R + rA)\\
                          \vr+r'A  & \mapsto d\vr+rA.
 \end{align*}
 
 Recall Definition \ref{def:divisor_flow} and Equation \eqref{eq:DphiDe}. 
 For a tropical curve $X$ with model $(\Gamma,\ell)$, the divisors  $\D_{d\vr}:=\D_{d\vr, r, D_{\Gamma,\R}}$ and $\D_{\vr}:=\D_{\vr,r',D_{\Gamma,\R'}}$ are equivalent. Indeed, we have
 \[
 \D_{d\vr}-\D_{\vr}=\sum_{e\in E(\ora{\G})}\vr(e)\left(-(d-1)\sigma(e)+dq_{e,r}-q_{e,r'}\right),
 \]
 where $q_{e,r}$ and $q_{e,r'}$ denote the points on $e\in E(\ora{\G})$ at distance $\frac{\ell(e)}{r}$ and $\frac{\ell(e)}{r'}$ from $\sigma(e)$, respectively. Thus Lemma \ref{lem:segment} (i) implies that the divisor $\D_{d\vr}-\D_{\vr}$ on $X$ is principal. This in turn induces a map of cones $\kappa_{\Gamma, \vr+r'A}\to \kappa_{\Gamma,d\vr+rA}$ taking a point in the cone $\kappa_{\Gamma, \vr+r'A}$ with coordinates $(x_e)_{e\in E(\G)}$ to the point of the cone $\kappa_{\Gamma,d\vr+rA}$ with the same coordinates $(x_e)_{e\in E(\G)}$.

The maps of cones $\kappa_{\Gamma, \vr+r'A}\to \kappa_{\Gamma,d\vr+rA}$ takes the point parametrizing a pair $(X,[\D_{\vr}])$ to the point parametrizing a pair $(X,[\D_{d\vr}])$. This follows by Equation \ref{eq:ye}, since, using Equation \eqref{eq:d_gcd}, we have that 
\[
\left|\frac{r'}{\gcd(r',\vr(e)+r'A)}\right|x_e=\left|\frac{r}{\gcd(r,d\vr(e)+rA}\right|x_e.
\]
Notice that this also proves that the map $\kappa_{\Gamma, \vr+r'A}\to \kappa_{\Gamma,d\vr+rA}$ is compatible with the forgetful maps onto $M_{g,n}^{\trop}$. 
Thus we get a map of generalized cone complexes which, by the description of the induced map between the cones, is a homeomorphism onto its image. This finishes the proof. 
\end{proof}

\begin{prop}\label{prop:multiplid}
Consider integers $r'$ and $d$ such that $r=dr'$, for some integer $d$. Let $\R=(a_1,\ldots, a_n,\ell)$ such that $a_1+\ldots+a_n+(2g-2)\ell\in rA\subset r'A$. Then there is a surjective natural map of generalized cone complexes
\[
 \Roots^{r,\trop}_g(A,\R)\to \Roots^{r',\trop}_g(A,\R)
\]
over $M_{g,n}^{\trop}$, 
taking the class of a pair $(X,[\D])$ to the class of the pair $(X, [d\D])$.
\end{prop}
\begin{proof}
   By Proposition \ref{prop:rankmap}, we have a natural surjective map
   \begin{align}
       \flowp_g(A/rA, \R+rA)&\to \flowp_g(A/r'A, \R+r'A)\label{eq:natual-surjective}\\
         \vr+rA&\mapsto \vr+r'A \nonumber.
   \end{align}
   For a tropical curve $X$ with model $(\Gamma,\ell)$, the divisors $d\D_{\vr,r}$ and $\D_{\vr,r'}$ are equivalent. Indeed, by Definition \ref{def:divisor_flow} and Equation \ref{eq:DphiDe}, we have
   \[
   d\D_{\vr,r}-\D_{\vr,r'} = \sum_{e\in E(\ora{\G})}\vr(e)\left(-(d-1)\sigma(e)+dq_{e,r}-q_{e,r'}\right),   \]
    where $q_{e,r}$ and $q_{e,r'}$ denote the points on $e\in E(\ora{\G})$ at distance $\frac{\ell(e)}{r}$ and $\frac{\ell(e)}{r'}$ from $\sigma(e)$, respectively.
   Thus Lemma \ref{lem:segment} implies that $d\D_{\vr,r}$ is equivalent to $\D_{\vr,r'}$.
    This induces a function $\kappa_{\Gamma,\vr+rA}\to\kappa_{\Gamma, \vr+r'A}$ that takes a point in the cone $\kappa_{\Gamma,\vr+rA}$ with coordinates $(x_e)_{e\in E(\Gamma)}$ to the point of the cone $\kappa_{\Gamma, \vr+r'A}$ with coordinates $(y_e)_{e\in E(\Gamma)}$, where
    \begin{equation}\label{eq:yexe}
    y_e=x_e\left | \frac{r\gcd(r',\vr(e)+r'A)}{r'\gcd(r,\vr(e)+rA)} \right |.
    \end{equation}
    Since the number multiplying $x_e$ in Equation \ref{eq:yexe} is a positive integer, this map is a map of cones that induces a map  of generalized cone complexes $\Roots^{r,\trop}_g(A,\R)\to \Roots^{r',\trop}_g(A,\R)$ over $M_{g,n}^{\trop}$. 
    The surjectivity of this map follows from the surjectivity of the map in Equation \ref{eq:natual-surjective}.
\end{proof}

\begin{prop}\label{prop:ring-change}
Let $f\col A\to B$ be a homomorphism of abelian groups and assume that the multiplication by $r$ is injective in $A$ and $B$. Let $\R=(a_1,a_2,\dots,a_n,\ell)$ be an $A$-ramification sequence. Then there is a natural map of generalized cone complexes 
\[
f_*\col \Roots^{r,\trop}_g(A,\R)\to \Roots^{r,\trop}_g(B,f(\R))
\]
over $M_{g,n}^{\trop}$, 
taking the class of a pair $(X,[\D])$ to class of the pair $(X,[f\circ\D)])$. If $f$ is injective  (respectively, surjective), then $f_*$ is injective (respectively, surjective). Here we denote by $f(\R)$ the $B$-ramification sequence $f(\R):=(f(a_1),f(a_2),\ldots, f(a_n),f(\ell))$.
\end{prop}

\begin{proof}
   By Proposition \ref{prop:rankmap}, we have a natural map $\flowp_{g}(A/rA, \R+rA)\to \flowp_g(B/rB, f(\R)+rB)$. Moreover, we have an induced function  $\kappa_{\Gamma,\vr+rA}\to\kappa_{\Gamma, f\circ\vr+rB}$ that takes a point in the cone $\kappa_{\Gamma,\vr+rA}$ with coordinates $(x_e)_{e\in E(\G)}$ to the point of the cone $\kappa_{\Gamma, f\circ\vr+rB}$ with coordinates $(y_e)_{e\in E(\Gamma)}$, where 
   \begin{equation}\label{eq:yexe2}
   y_e=x_e \left|\frac{\gcd (r,f(\vr(e))+rB)}{\gcd(r,\vr(e)+rA)}\right|.
   \end{equation}
   Notice that the number multiplying $x_e$ in Equation \ref{eq:yexe2} is a positive integer. Thus the above function is a map of cones inducing a map of generalized cone complexes $f_*\col \Roots^{r,\trop}_g(A,\R)\to \Roots^{r,\trop}_g(B,f(\R))$ over $M_{g,n}^{\trop}$. Clearly, the point parametring $(X,[\D])$ is sent to the point parametrizing $(X,[f\circ\D])$.  By Proposition \ref{prop:rankmap}, if $f$ is injective (respectively, surjective), then $f_*$ is injective (respectively, surjective).
\end{proof}

\begin{remark}
 Given a ramification sequence $\R=(a_1,\ldots, a_n,\ell)$ we denote by $\mathcal{M}_{g,n}^{r}(\R)$ the stack parametrizing pairs $(C,L)$ where $C$ is a pointed smooth curve of genus $g$ and $L$ is an invertible sheaf on $C$ such that $L^{r}=\omega^{\ell}_C(a_1P_1+\ldots+a_nP_n)$. 
 
The maps in Propositions \ref{prop:homeoRprime} and \ref{prop:multiplid}
have geometric analogues. Fix $r=dr'$. If $\R=d\R'$ for some $\R'$ then we have  a natural map
\begin{align*}
    \mathcal{M}_{g,n}^{r'}(\R')\to \mathcal{M}_{g,n}^{r}(\R)\\
     (C, \mathcal{L})  \mapsto (C,\mathcal{L}).
\end{align*}
This map is an isomorphism onto its image, and in fact its image is a union of irreducible components of $\mathcal{M}_{g,n}^{r}(\R)$.
We also have the natural map
\begin{align*}
    \mathcal{M}_{g,n}^{r}(\R)\to \mathcal{M}_{g,n}^{r'}(\R)\\
     (C, \mathcal{L})\mapsto (C,\mathcal{L}^{\otimes d}),
\end{align*}
which is an étale surjective map.
\end{remark}

\section{The tropicalization of the moduli space of roots}

In this section we consider the moduli space of roots on algebraic curves constructed by Jarvis and by Caporaso, Casagrande and Cornalba. We relate its Berkovich skeleton with the moduli space of roots on tropical curves constructed in the previous section.

Let $k$ be an algebraically closed ﬁeld.
We will always assume our curves to be nodal curves over $k$. We will denote by $g$ the (arithmetic) genus of a curve. A $n$-pointed curve is a curve together with $n$ distinct smooth points.  Recall that a pointed curve is \emph{stable} if so is its dual graph.  

Let $C$ be a curve. We denote by $C^{\sing}$ the singular locus of  $C$. Given a vertex $v$ of the dual graph of $C$, we let $C_v$ be the component of $C$ corresponding to $v$. 
Consider a subset $\Delta\subset C^{\sing}$ of nodes of $C$. We let $C_\Delta^\nu$ be the partial normalization of $C$ at $\Delta$. A \emph{blowup} of $C$ at the set $\Delta$ is a curve $C_\Delta$ which is the union of $C_\Delta^\nu$ and smooth connected rational components $E_p$ (called \emph{exceptional components}), for $p\in \Delta$, such that $E_p\cap C_\Delta^\nu$ is equal to the set of smooth points of $C_\Delta^\nu$ lying over the node $p\in \Delta$ of $C$. There is a natural map $\pi_\Delta\col C_\Delta\ra C$ contracting the exceptional components.

\subsection{Nets of limit $r$-roots}
\label{sec:nets}

Throughout this section we let $r$ be an integer.
We begin by recalling the definition of \emph{limit $r$-root} introduced in \cite{CCC}. 

\begin{definition}\label{def:r-root}(\cite[Definition 2.1.1]{CCC}). 
Let $C$ be a  curve and $N$ be an invertible sheaf on $C$ whose degree is a multiple of $r$.
A \emph{limit $r$-root} of $N$ is a triple $(C_\Delta,L,\alpha)$, where $\Delta$ is a subset of nodes of $C$, $L$ is an invertible sheaf on $C_\Delta$, and  $\alpha\col L^r\to \pi_\Delta^*(N)$ is a homomorphism, such that
\begin{enumerate}
  \item the degree of $L$ over every exceptional component of $C_\Delta$ is $1$;
  \item if we set  $\pi^\nu_\Delta:=\pi_\Delta|_{C^\nu_\Delta}$, then $\alpha|_{C^\nu_\Delta}$ factors through an isomorphism
  \[
  L|_{C^\nu_\Delta}^r\stackrel{\cong}{\to} (\pi^\nu_\Delta)^*(N)\left(-\sum_{p\in C^{\sing}}\sum_{q\in(\pi^\nu_\Delta)^{-1}(p)} u_q q\right),
  \]
  where $u_q$ are integers satisfying the following properties:
  \begin{enumerate}
       \item $u_q\in\{0,1,\dots,r-1\}$
      \item $u_q=0$ for $q\in \pi_\Delta^{-1}(p)$ if and only if $p\not\in \Delta$;
      \item $\sum_{q\in(\pi^\nu_\Delta)^{-1}(p)}u_q\equiv 0 \text{ mod } (r)$.
  \end{enumerate}  
\end{enumerate}
  For every $p\in C^{\sing}$, we call $\{u_q:q\in (\pi^{\nu}_\Delta)^{-1}(p)\}$ the \emph{set of vanishing orders at $p$} of the limit $r$-root.
\end{definition}

If $\mathcal C\to T$ is a family of nodal curves and $\mathcal N$ is an invertible sheaf on $\mathcal C/T$, there is a natural definition of a limit $r$-root $\mathcal L$ of $\mathcal N$, obtained by asking that the restriction $\mathcal L|_{C_t}$ of $\mathcal L$ to the fiber $C_t$ over any closed point $t\in T$ is a limit $r$-root of the restiction $\mathcal N|_{C_t}$ (see \cite[Section 2.1]{CCC}).

An equivalent definition of limit $r$-root using torsion-free rank-$1$ sheaves is given by Jarvis.

\begin{definition}(\cite[Definition 2.1.2]{Jarvis})\label{def:Jarvis}
Let $C$ be a  curve and $N$ an invertible sheaf on $C$ whose degree is a multiple of $r$.
   A \emph{limit $r$-root} of $N$ (in the sense of Jarvis) on $C$ is a pair  $(F,c)$, where $F$ is a torsion-free rank-$1$ sheaf over $C$ and $c$ is a homomorphism $c\colon F^r\to N$ whose cokernel has rank $r-1$ at every node of $C$ at which $F$ is not locally free. We denote by $\Delta(F)$ the locus of nodes of $C$ where the sheaf $F$ fails to be locally free.
\end{definition}

\begin{remark}
The two notions of limit $r$-roots are equivalent (see \cite[Theorem 4.2.3]{CCC}). The equivalence takes a triple $(C_\Delta,L,\alpha)$ as in Definition \ref{def:r-root} to the limit $r$-root $(F,c)$ (in the sense of Jarvis) with $F=(\pi_\Delta)_*(L)$ and $c=(\pi_\Delta)_*(\alpha)$.
\end{remark}

\begin{remark}
\label{rem:flow_fixed_at_node}
  Let $f\col \mathcal{C}\to T$ be a family of curves with $T$ connected and $\sigma\col T\to \mathcal{C}$ be a section of $f$ whose image lands in the singular locus of $f$. Let $\mathcal N$ be an invertible sheaf on $\mathcal C/T$.  Given a limit $r$-root of $\mathcal N$, 
  the set of vanishing orders at $\sigma(t)$ of $\mathcal L|_{C_t}$ for $t$ varying in $T$, is independent of $t$. This follows from \cite[Propositions 3.3.1 and 5.4.3]{Jarvis} (see also \cite[Section 3.3]{CCC}).
\end{remark} 

Next, consider an invertible sheaf $\mathcal N$ on the universal family over $\overline{\mathcal M}_{g,n}$. By \cite[Theorem 2.4.1 and 4.2.1]{CCC} there is a Deligne-Mumford stack $\overline{\mathcal S}^r_{g,n}(\mathcal N)\to \overline{\mathcal M}_{g,n}$ parametrizing isomorphism classes of limit $r$-roots of $\mathcal N|_C$, for $C$ varying through all $n$-pointed stable curves. 
The stack $\overline{\mathcal S}^r_{g,n}(\mathcal N)$ could fail to be normal (this happens if and only if $r$ is not prime). Jarvis gave a modular description of its normalization in terms of nets of limit roots. We recall some key results here, referring to \cite{JarvisNet} for more details.
  
  \begin{definition}\label{def:net}(\cite[Definition 2.4]{JarvisNet})
  Let $C$ be a nodal curve and $N$ an invertible sheaf on $C$. 
  A \emph{net of limit $r$-roots of $N$} is a collection $\net = \{F_d,c_{d,d'}\}$ for positive integers $d,d'$ dividing $r$ where each $F_{d}$ is a torsion free rank-$1$ sheaf and $c_{d,d'}$ are maps $c_{d,d'}\colon F_{d}^{d/d'}\to F_{d'}$  such that $(F_{d}, c_{d,d'})$ is a $d/d'$-root of $F_{d'}$, and such that the natural diagrams commute. 
     
     We stress that, given a net of limit $r$-roots $\{F_d,c_{d,d'}\}$, the pair $(F_r,c_{r,1})$ is a limit $r$-root of $N$ (in the sense of Jarvis). We write $\Delta(\net):=\Delta(F_r)$ (see Definition \ref{def:Jarvis}).
  \end{definition}

\begin{definition}
Let $r$ be an integer. Let $C$ be a curve and $N$  an invertible sheaf on $C$ whose degree is a multiple of $r$.
Consider a limit $r$-root $(C_\Delta,L,\alpha)$ of $N$ with vanishing order $\{u_q:q\in (\pi^{\nu}_\Delta)^{-1}(p)\}$ at $p\in C^{\sing}$. The \emph{combinatorial type} of $(C_\Delta,L,\alpha)$  is the pair $(\Gamma,\vr)$, where $\Gamma$ is the dual graph of $C$  and $\varphi$ is the flow on $\Gamma$ taking $e\in \ora{E}(\Gamma)$ to $\vr(e)=u_q$ for $q\in (\pi_\Delta^{\nu})^{-1}(p_e)$ and $q\in C_{\sigma(e)}$ where $p_e$ is the node in $C^{sing}$ corresponding to $e$.
\par 
  If $\net = \{F_d,c_{d,d'}\}$ is a net of limit $r$-root of $N$, then the \emph{combinatorial type} of the pair $(C,\net)$ is the combinatorial type of the limit $r$-root corresponding to $(F_r,c_{r,1})$.
\end{definition}

   The definition of a net of limit $r$-roots for families of nodal curves requires some additional technical conditions, see \cite[Definition 2.9]{JarvisNet}. These conditions are needed for the moduli space of nets to be well behaved. We will not give the precise definition, as we will only need its consequences, that are described in Equations \eqref{eq:local} and \eqref{eq:forgetful-map} below.

Consider an invertible sheaf over the universal family $\mathcal C\ra \overline{\mathcal M}_{g,n}$ over $ \overline{\mathcal M}_{g,n}$. Recall that every such sheaf is of the form $\mathcal N(\mathcal R)=\omega_{\C}^\ell(a_1\Sigma_1+\ldots +a_n\Sigma_n)$ for some $\mathbb Z$-ramification sequence $\R=(a_1,\ldots, a_n,\ell)$, where $\Sigma_i$ are the images of the sections of the universal family.

By \cite[Theorem 2.3]{JarvisNet},
for every $\mathbb Z$-ramification sequence $\R=(a_1,\ldots, a_n,\ell)$ such that $r$ divides $a_1+\ldots+a_n+\ell(2g-2)$, 
there exists a smooth Deligne-Mumford stack $\overline{\mathcal M}_{g,n}^r(\mathcal R)$ over 
$\overline{\mathcal M}_{g,n}$ parametrizing isomorphism classes of nets of limit $r$-roots of $\mathcal N(\mathcal R)$.

 Consider a point $[C, \net]$ of $\overline{\M}_{g,n}^r(\R)$ parametrizing the class of a pair $(C, \net)$, where $C$ is a stable curve and $\net$ is a net of limit $r$-roots of $\mathcal{N}(\R)|_C$. Let $(\Gamma,\varphi)$ be the combinatorial type of $(C, \net)$. By \cite[Theorem 2.2]{JarvisNet}, the completion of the local ring of $\overline{\M}_{g,n}^r(\R)$ at $[C, \mathfrak{F}]$ is given by
\begin{equation}\label{eq:local}
\widehat{\mathcal{O}}_{\overline{\M}_{g,n}^r(\R),[C, \mathfrak{F}]} \cong k[[\tau_e]]_{e\in \Delta(\mathfrak{F})}\otimes k[[t_e]]_{e\in E(\Gamma)\setminus\Delta(\mathfrak{F})}\otimes k[[t_1,\ldots, t_{3g-3+n-|E(\Gamma)}]].
\end{equation}

 The space $\overline{\mathcal M}_{g,n}^r(\mathcal R)$ is endowed with a natural forgetful map $\pi^r_{g,\R}\col \overline{\mathcal M}_{g,n}^r(\mathcal R)\to \overline{\mathcal M}_{g,n}$. 
Write the completion of the local ring of $\overline{\M}_{g,n}$ at the point $[C]$ as 
\[
\widehat{\mathcal{O}}_{\overline{\M}_{g,n},[C]}\cong  k[[t_e]]_{e\in E(\Gamma)}\otimes k[[t_1,\ldots, t_{3g-3+n-|E(\Gamma)}]].
\]
 Then the natural forgetful map $\overline{\M}_{g,n}^r(\R)\to \overline{\M}_{g,n}$ is given locally by 
\begin{equation}
\label{eq:forgetful-map}
\begin{aligned}
\widehat{\mathcal{O}}_{\overline{\M}_{g,n},[C]}&\to\widehat{\mathcal{O}}_{\overline{\M}_{g,n}^r(\R),[C, \mathfrak{F}]}\\
t_e&\mapsto \tau_e^{s_e},\quad \text{if $e\in \Delta(\mathfrak{F})$ }\\
t_e& \mapsto t_e,  \quad\text{ if $e\notin \Delta(\mathfrak{F})$ }\\
t_i& \mapsto t_i, \quad \forall i,
\end{aligned}
\end{equation}
where $s_e := r/\gcd(r, \varphi(e))$. These local equations follow from the technical conditions defining a family of nets of limit $r$-roots (see \cite[Definition 2.9]{JarvisNet} for the conditions and \cite[Theorem 2.2]{JarvisNet} for the results).

  By \cite[Theorem 2.3]{JarvisNet}, we have that $\overline{\mathcal M}_{g,n}^r(\mathcal R)$ is the normalization of $\overline{\mathcal S}^r_{g,n}(\mathcal R)$, where the normalization map $ \overline{\mathcal M}_{g,n}^r(\mathcal R)\to \overline{\mathcal S}^r_{g,n}(\mathcal R)$ takes a point $[C, \net]$ in $\overline{\mathcal M}_{g,n}^r(\mathcal R)$, with $\net = \{F_d,c_{d,d'}\}$, to the point of $\overline{\mathcal S}^r_{g,n}(\mathcal R)$ representing $(F_r,c_{r,1})$.

We let $\mathcal M_{g,n}^r(\R)$ be the open dense subset of $\overline{\mathcal M}_{g,n}^r(\R)$ parametrizing nets of limit roots over smooth curves. The  open embedding $\mathcal M_{g,n}^r(\R)\subset \overline{\mathcal M}_{g,n}^r(\R)$ 
gives rise to a stratification of $\overline{\mathcal M}_{g,n}^r(\mathcal N)$, which we can write as follows:
\[
\overline{\mathcal M}_{g,n}^r(\R=\sqcup_\Gamma (W_{\Gamma,1}\sqcup \dots\sqcup W_{\Gamma,m_\Gamma}),
\]
for $\Gamma$ varying over stable genus-$g$ graphs with $n$-legs and
where the loci $W_{\Gamma,1},\dots,W_{\Gamma,m_\Gamma}$ are the irreducible components of $\pi_{g,\R}^{-1}(\mathcal M_\Gamma)$, where $\M_{\Gamma}\subset \ol{\M}_{g,n}$ is the locus of $\overline{\M}_{g,n}$ parametrizing stable curves with dual graph isomorphic to $\Gamma$. Of course, we have $\mathcal M_{g,n}^r(\mathcal N)=\pi_{g,\R}^{-1}(\mathcal M_{g,n})$.

\begin{prop}\label{prop:strata}
If $W$ is a stratum of $\overline{\mathcal M}_{g,n}^r(\mathcal N)$, then the combinatorial type of every net of $r$-root parametrized by  a  point $w\in W$ is independent of $w$. Moreover, if $W'$ is another stratum of $\overline{\mathcal M}_{g,n}^r(\mathcal N)$ such that $W$ is contained in the closure of $W'$ in $\overline{\mathcal M}_{g,n}^r(\mathcal N)$ and if we let $(\Gamma_W,\vr_W)$ and $(\Gamma_{W'},\vr_{W'})$ the combinatorial types of a net of limit $r$-root parametrized by a  point of $W$ and $W'$, respectively, then $(\Gamma_W,\vr_W)$ specializes to $(\Gamma_{W'},\vr_{W'})$. 
\end{prop}
\begin{proof}
The stratum $W$ can be written as $W=W_{\Gamma, j}$ for some $j$. Hence, for every point $[C]\in W$ the dual graph of $C$ is isomorphic to $\Gamma$.
   Since $W$ is irreducible, it is sufficient to prove that the combinatorial type of every net of limit $r$-root parametrized by a point $w\in W$ is isomorphic to the  combinatorial type of the net of limit $r$-root parametrized by the generic point of $W$. In turn, we can argue locally étale, and hence we can conclude using Remark \ref{rem:flow_fixed_at_node}, since we have a section for each node.\par 
    The second statement follows from the first and Remark \ref{rem:flow_fixed_at_node}.
\end{proof}

Consider a stratum $W$ of $\overline{\mathcal M}_{g,n}^r(\mathcal N)$. \emph{The combinatorial type} of $W$ is the combinatorial type of a net of limit $r$-root parametrized by any point of $W$. This definition is well-defined by Proposition \ref{prop:strata}. 
Notice that, if $(\Gamma_W,\vr_W)$ is the  combinatorial type of $W$, 
we have that $W$ lies over a stratum $\M_{\Gamma}$ of $\overline{\M}_{g,n}$ if and only if $\Gamma_W=\Gamma$.

 Consider a stable genus-$g$  graph $\Gamma$ with $n$ legs. We define
 \[
 \widetilde{\M}_{\Gamma}=\prod_{v\in V(\Gamma)}\mathcal M_{w(v),\deg_\Gamma(v)+|L^{-1}(v)|}.
 \]
 Recall that we have a natural quotient map $\widetilde{\M}_{\Gamma}\to \M_{\Gamma} $ such that $\M_{\Gamma}\cong\widetilde{\M}_{\Gamma}/\Aut(\Gamma)$. Notice that we can view  $\widetilde{\M}_{\Gamma}$ as the space parametrizing pairs $(C,\tau)$ where $C$ is a stable nodal curve with dual graph isomorphic to $\Gamma$ and $\tau\col \Gamma_C\to \Gamma$ is such an isomorphism.

  Let $W$ be a stratum of $\overline{\M}_{g,n}^r(\mathcal{N})$ of combinatorial type $(\Gamma,\vr)$. The product $W\times_{\mathcal M_\Gamma} \widetilde{\mathcal M}_{\Gamma}$ parametrizes triples $(C,\net,\tau)$ where $\net$ is a net of limit $r$-roots of $\mathcal N|_C$ and $\tau\col \Gamma_C\to \Gamma$ is an isomorphism. Notice that we have an action of $\Aut(\Gamma)$ on $W\times_{\mathcal M_\Gamma} \widetilde{\M}_\Gamma$. 
  All the connected components of $W\times_{\M_{\Gamma}} \widetilde{\M}_{\Gamma}$ have the same stabilizer, since the map $W\times_{\M_{\Gamma}} \widetilde{\M}_{\Gamma}\to W$ is Galois.

\begin{prop}
\label{prop:stabilizer}
Let $W$ be a stratum of $\overline{\M}_{g,n}^r(\mathcal{\mathcal{N}})$ of combinatorial type $(\Gamma,\vr)$. Then the stabilizer of the connected components of $W\times_{\mathcal M_\Gamma} \widetilde{\M}_\Gamma$ is contained in $\Aut(\Gamma,\vr)$. 
\end{prop}
\begin{proof}
   Consider two points $(C,\net,\tau_1)$ and $(C,\net,\tau_2)$ in the same connected component of $W\times_{\M_\G} \widetilde{\M}_{\Gamma}$, for some $\tau_1,\tau_2\in \Aut(\Gamma)$. Let $\mathcal C\to \mathcal M_\Gamma$ be the restriction of the universal family over $\overline{\mathcal M}_{g,n}$, and set $\widetilde{\mathcal C}=\mathcal C\times_{\mathcal M_\Gamma}\widetilde{\mathcal M}_\Gamma$. For every edge $e\in E(\Gamma)$ we have a  section $\widetilde{\mathcal M}_\Gamma\to \widetilde{\mathcal C}$ whose image is the locus of the node corresponding to $e$. This gives rise to sections $\sigma_e\col W\times_{\mathcal M_\Gamma} \widetilde{\M}_{\Gamma}\to 
   W\times_{\mathcal M_\Gamma} \widetilde{\mathcal C}$. By Remark \ref{rem:flow_fixed_at_node}, we have that the flow of the combinatorial type of $\net$ has the same value on the edges $\tau_1^{-1}(e)$ and $\tau_2^{-1}(e)$. This means that $\tau_1^{-1}\circ\tau_2\in \Aut(\Gamma,\vr)$, which proves that the stabilizer of the connected components of $W\times_{\mathcal M_\Gamma} \widetilde{\M}_\Gamma$ is contained in $\Aut(\Gamma,\vr)$.
\end{proof}

\subsection{The skeleton of the moduli space of nets}

Let $\mathcal Y$ be a separated, connected, proper Deligne-Mumford stack over an algebraically closed field $k$. The \emph{Berkovich analytification} $\Y^{an}$ of $\Y$ is an analytic stack. We work with the topological space underlying $\Y^{an}$, whose points are morphisms $\Spec(K)\to \Y$ that extends to $\Spec(R)\to \Y$, where $K$ is a non Archimedean valued field extension of the trivially valued field $k$ and $R$ is the valuation ring of $K$, up to equivalence by further valued field extensions  (see \cite{Be} and \cite[Section 3]{U17}). We abuse notation and use $\Y^{\an}$ for both the stack and the topological space underlying it. One can choose a representative $\Spec(K)\to \Y$ with $K$ complete for each point in $\Y^{an}$.  
  
  Let $U\subset \Y$ be a toroidal embedding of Deligne-Mumford stacks. Define the sheaves $\mon_{\Y}$ and $\eff_{\Y}$ as the \'etale sheaves over $\Y$ such that, for every \'etale morphism $V\to \Y$ from a scheme $V$, we have that $\mon_\Y(V)$ (respectively, $\eff_\Y(V)$) is the group of Cartier divisors on $V$ (respectively, the submonoid of effective Cartier divisors on $V$) supported on $V\setminus U_V$, where $U_V=U\times_\Y V$. \par
	For each stratum $W\subset \Y$ and a point $w\in W$, we have an action of the \'etale fundamental group $\pi_1^{\et}(W,w)$ on the stalk $\mon_{\Y,w}$ preserving $\eff_{\Y,w}$. The \emph{monodromy group} $H_W$ is defined as the image of $\pi_1^{\et}(W,w)$ in $\Aut(\mon_{\Y,w})$.\par
   For each stratum $W\subset \Y$, there is an associated extended cone
\[
\ol{\sigma}_W:=\Hom_{\text{monoids}}(\eff_{\Y,w},\ol{\RR}_{\geq0}),
\]
where $\ol{\RR}=\RR\cup\{\infty\}$.
 The \emph{skeleton} of $\Y$ is the extended generalized cone complex:
\[
\ol{\Sigma}(\mathcal Y):=\lim_{\lra}\ol{\sigma}_W,
\]
where the arrows $\ol{\sigma}_W\to\ol{\sigma}_{W'}$ are given by the inclusions $W'\subset \ol{W}$, where $\ol W$ is the closure of $W$ in $\Y$, and by the monodromy group $H_W$ when $W=W'$. For more details, we refer to \cite{T} and \cite[Section 6]{ACP}.

There is a retraction map ${\bf p}_{\mathcal Y}\col \Y^{\an}\ra \ol{\Sigma}(\mathcal Y)$ defined as follows. Let $\psi\col \Spec(R)\to \Y$ be a point in $\Y^{\an}$, with $R$ complete. Let $w\in \Y$ be the image of the closed point in $\Spec(R)$ and $W$ be the stratum of $\Y$ containing $w$. We have a chain of maps
\begin{equation}
\label{eq:p_retraction}
\eff_{\Y,w}\stackrel{\epsilon}{\lra} \widehat{\O}_{\Y,w}\stackrel{\psi^\#}{\lra} R\stackrel{\nu_R}{\lra} \ol{\RR}_{>0},
\end{equation}
where $\epsilon$ is the map that takes an effective divisor to its local equation and $\nu_R$ is the valuation of $R$. The composition is a morphism of monoids. We  define ${\bf p}_{\Y}(\psi)\in \ol{\Sigma}(\Y)$ as the equivalence class of $\nu_R\circ\psi^\#\circ\epsilon\in\ol{\sigma}_{W}$ (see \cite[Section 6]{ACP} for details, in particular  \cite[Propositions 6.1.4. and 6.2.6]{ACP}).

The inclusions $\mathcal M_{g,n}\subset \ol{\mathcal M}_{g,n}$ and $\M_{g,n}^r(\R)\subset  \overline{\M}_{g,n}^r(\R)$ are toroidal embeddings of Deligne-Mumford stacks  (see \cite[Section 3.3]{ACP} and Proposition \cite[Theorem 2.4.2]{JarvisNet}). We let $\ol{\Sigma}(\ol{\mathcal {M}}_{g,n})$ and  $\ol{\Sigma}(\ol{\M}_{g,n}^{r}(\R))$  be the skeleta of $\ol{\mathcal  M}_{g,n}$ and $\ol{\mathcal  M}_{g,n}^r(\R)$, respectively.

The forgetful map $\pi_{g,\mathcal R}\colon \ol{\mathcal  M}_{g,n}^r(\R)\ra \ol{\mathcal M}_{g,n}$ induces a natural map 
\[
\pi^{\an}_{g,\mathcal R}\col\ol{\mathcal M}_{g,n}^r(\R)^{\an} \ra \ol{\mathcal M}_{g,n}^{\an}.
\]
Since  $\ol{\M}_{g,n}^r(\R)$ and $\ol{\M}_{g,n}$ are proper and the forgetful map is toroidal (by Equation \ref{eq:forgetful-map}), we have, by \cite[Proposition 6.1.8]{ACP}, that the map $\pi^{\an}_{g,\mathcal R}$ restricts to a map of generalized extended cone complexes 
\[
\ol{\Sigma}(\pi^{\an}_{g,\mathcal R})\col  \ol{\Sigma}(\ol{\M}_{g,n}^r(\R)) \ra  \ol{\Sigma}(\ol{\mathcal M}_{g,n}).
\]
Moreover, we have the following functorial property with respect to the retraction map:
\begin{equation}\label{prop:funct}
{\bf p} _{\ol{\M}_{g,1}}\circ\pi^{\an}_{g,\mathcal R}=\ol{\Sigma}(\pi^{\an}_{g,\mathcal R})\circ{\bf p} _{\ol{\M}_{g,n}^r(\R)}.
\end{equation}



We define the \emph{tropicalization map}
\begin{equation}\label{eq:tropmap}
\trop_{\ol{\M}_{g,n}^r(\R)}\col \overline{\M}_{g,n}^r(\R)^{\an}\to \overline{\Roots}_{g}^{r,\trop}(\R)
\end{equation}
as follows.
 Fix a point $\psi\col \Spec(R)\to \overline{\M}_{g,n}^r(\R)$ of $\overline{\M}_{g,n}^r(\R)^{\an}$ with $R$ complete and let $w=[C,\net]$ be  the image  of the closed point of $\Spec(R)$ via $\psi$ (this is the class of a pair consisting of a curve $C$ and a net $\net$ of limit $r$-roots on $C$). We get a ring homomorphism
\[
\psi^\#\col \widehat{\O}_{\overline{\M}_{g,n}^r(\R),w}\to R.
\]
Let $(\Gamma,\vr)$ be the combinatorial type of $(C,\net)$. Recall Equation \eqref{eq:local}. We define a length function $\ell\col E(\Gamma)\to \ol{\RR}_{>0}$ given by 
\begin{equation}
\label{eq:trop_length_function}
    \ell(e)=\begin{cases}          s_e\nu_R(\psi^\#(\tau_e)) & \text{ if } e \in \Delta(\net), \\
          \nu_R(\psi^\#(t_e)) & \text{ if } e \notin \Delta(\net).
    \end{cases}
\end{equation}  
We define $\trop_{\ol{\M}_{g,n}^r(\R)}(\psi)=[X,[\D_\vr]]$, where $X$ is the (possibly extended) tropical curve with model $(\Gamma,\ell)$. Note that the point $(x_e)\in \kappa_{(\Gamma,\vr)}$ that corresponds to $[X,[\D_{\vr}]]$ is given by 
\begin{equation}
\label{eq:trop_length_function1}
    x_e=\begin{cases}
          \nu_R(\psi^\#(\tau_e)) & \text{ if } e \in \Delta(\net), \\
          \nu_R(\psi^\#(t_e)) & \text{ if } e \notin \Delta(\net).
    \end{cases}
\end{equation}

\begin{corollary}\label{cor:monodromy}
Let $W$ be a stratum of $\overline{\M}_{g,n}^r(\mathcal{R})$ of combinatorial type $(\Gamma,\vr)$.
  Then the monodromy of the stratum $W$ is contained in $\Aut(\Gamma,\vr)$.
\end{corollary}
\begin{proof}
   The sheaf $\text{Eff}_{\overline{\M}_{g,n}^r(\mathcal{R})}$ is trivial on $W\times_{\M_{\Gamma}} \widetilde{\M}_{\Gamma}$, because it is trivial on $\widetilde{\M}_{\Gamma}$ (see \cite[proof of Propostion 7.2.1]{ACP}). In particular, this sheaf is trivial on every connected component of $W\times_{\M_{\Gamma}} \widetilde{\M}_{\Gamma}$, so the monodromy of $W$ coincides with the stabilizer of one (every) connected component. The result follows from Proposition \ref{prop:stabilizer}.
\end{proof}

The following is the main result of the section.
Recall Definition \ref{def:moduli-root}.
\begin{theorem}
\label{thm:tropj}
There is a morphism of extended generalized cone complexes
\[
\Phi_{\ol{\M}_{g,n}^r(\R)}\col \ol{\Sigma}(\ol{\M}_{g,n}^r(\R))\to \ol{\Roots}_{g}^{r,\trop}(\R),
\]
and a commutative diagram:
\begin{eqnarray*}
\SelectTips{cm}{11}
\begin{xy} <16pt,0pt>:
\xymatrix{
\ol{\M}_{g,n}^r(\R)^{\an} \ar@/^2pc/[rrrr]^{\trop_{\ol{\M}_{g,n}^r(\R)}} \ar[d]_{\pi^{\an}_{g,\mathcal R}} \ar[rr]^{{\bf p}_{\ol{\M}_{g,n}^r(\R)}\;}   
  && \ar[rr]^{{\Phi}_{\ol{\M}_{g,n}^r(\R)}} \ar[d]_{\ol{\Sigma}(\pi^{\an}_{g,\mathcal R})} \ol{\Sigma}(\ol{\M}_{g,n}^r(\R)) && \ol{\Roots}_{g}^{r,\trop}(\R)   \ar[d]_{\pi^{\trop}_{g,\mathcal R}} \\
\ol{\M}_{g,n}^{\an}  \ar@/_2pc/[rrrr]_{\trop_{\ol{\M}_{g,n}}} \ar[rr]^{{\bf p}_{\ol{ {\M}}_{g,n}}\;}            & &  \ar[rr]^{{\Phi}_{\ol{\M}_{g,n}}}         \ol{\Sigma}(\ol{{\M}}_{g,n})  && \ol{M}_{g,n}^{trop}  
 }
\end{xy}
\end{eqnarray*}
\end{theorem}
\begin{proof}
  Let $w=[C,\net]$ be a point in a stratum $W$ of $\ol{\M}_{g,n}^r(\R)$ of combinatorial type $(\Gamma,\vr)$. Then there is an isomorphism of monoids $\Eff_{\ol{\M}_{g,n}^r(\R),w}\to\mathbb{Z}_{\geq0}^{E(\Gamma)}$ via Equation \eqref{eq:local}, hence $\ol{\sigma}_W$ is isomorphic  to $\ol{\sigma}_{(\Gamma,\vr)}$. By \cite[Proposition 6.2.6]{ACP} and Corollary \ref{cor:monodromy}, we get a map 
\[
\ol{\Sigma}(\ol{\M}_{g,n}^r(\R))=\coprod_{W} \ol{\sigma}^\circ_W/H_W\to\coprod_{(\Gamma,\vr)}\ol{\sigma}^\circ_{(\Gamma,\vr)}/\Aut(\Gamma,\vr).
\]
Moreover, by Proposition \ref{prop:strata}, if $W$ and $W'$ are strata with combinatorial types $(\Gamma,\vr)$ and $(\Gamma',\vr')$, respectively, we have that $(\Gamma,\vr)\geq (\Gamma',\vr')$ in the poset $\flowp_g(\mathbb{Z}/r\mathbb{Z},\mathcal R)$  if, and only if,   $W$ is contained in the closure of $\overline{W'}$ in $\ol{\M}_{g,n}^r(\R)$. This induces a map of generalized cone complexes $\Phi_{\ol{\M}_{g,n}^r(\R)}$ from $\ol{\Sigma}(\ol{\M}_{g,n}^r(\R))$ to $\overline{\Roots}_{g}^{r,\trop}(\R)$.\par
    Let $\psi\col \Spec(R)\to \ol{\M}_{g,n}^r(\R)$ be a point in $\ol{\M}_{g,n}^r(\R)^{\an}$, with $R$ complete. Let $W$ be the stratum of $\ol{\M}_{g,n}^r(\R)$ containing the image $w$ via $\psi$ of the closed point in $\Spec(R)$. The set $E(\Gamma)$ can be seen as a monoid basis of the free monoid $\Eff_{\ol{\M}_{g,n}^r(\R),w}$. Write  $\trop_{\ol{\M}_{g,n}^r(\R)}(\psi)=[X,[\D_\vr]]\in\ol{\Roots}_{g}^{r,\trop}(\R)$ and let
    $(x_e)_{e\in E(\Gamma)}$ be the corresponding point in $\overline{\sigma}_{\Gamma,\vr}^{\circ}$. 
    
    By Equation \eqref{eq:p_retraction}, we have that ${\bf p}_{\ol{\M}_{g,n}^r(\R)}(\psi)$ is the point of $\overline{\sigma}^\circ_W$ given by the coordinates
    \[
    (\nu_R(\psi^\#(\tau_e)) , \nu_R(\psi^\#(t_{e'}))),
    \]
    for $e\in \Delta(\net)$ and $e'\in E(\Gamma)\setminus\Delta(\net)$.
    Clearly, this is the same point as the point $(x_e)_{e\in e(\Gamma)}$ given by Equation \eqref{eq:trop_length_function1}.
It follows that $\trop_{\ol{\M}_{g,n}^r(\R)}=\Phi_{\ol{\M}_{g,n}^r(\R)}\circ{\bf p}_{\ol{\M}_{g,n}^r(\R)}$.

The fact that the square in the left hand side of the diagram in the statement is commutative follows from Equation \eqref{prop:funct}.

Finally, we have that $\trop_{\ol{\M}_{g,n}}\circ\pi^{\an}(\psi)$ is the tropical curve $(\Gamma,\ell')$ where $\ell'(e)=\nu_R(\psi^\#t_e)$. On the other hand, $\pi^{\trop}\circ \trop_{\ol{\M}_{g,n}^r(\R)}(\psi)$ is the tropical curve $(\Gamma,\ell)$ where $\ell$ is as in Equation \ref{eq:trop_length_function}. Since $\nu_R(\psi^\#t_e)=s_e\nu_R(\psi^\#\tau_e)$ by Equation \eqref{eq:forgetful-map}, we deduce that $\ell'(e)=\ell(e)$ if $e\in \Delta(\net)$.  It is also clear that   $\ell'(e)=\ell(e)$ if $e\notin \Delta(\net)$, hence the tropical curves $(\Gamma,\ell')$ and $(\Gamma,\ell)$ are equivalent, thus
\[
\trop_{\ol{\M}_{g,n}}\circ\pi^{\an}_{g,\mathcal R}=\pi^{\trop}_{g,\mathcal R}\circ \trop_{\ol{\M}_{g,n}^r(\R)}.
\]
This concludes the proof.
\end{proof}

\begin{remark}
  Usually, the map ${\Phi}_{\ol{\M}_{g,n}^r(\R)}$ is not  an isomorphism, as there is still some combinatorial data missing to give a finer description of the stratification of $\overline{\M}_{g,n}^r(\R)$.  For the case  $\R=(0,\ldots,0,1)$, this is fixed by the sign function considered in \cite{CMP}.
\end{remark}

\bibliographystyle{amsalpha}
\bibliography{bibli.bib}

\medskip

\noindent{\small Alex Abreu \\
Universidade Federal Fluminense, Rua Prof. M. W. de Freitas, S/N\\ 
Niter\'oi, Rio de Janeiro, Brazil. 24210-201.\\
email: alex\_abreu@id.uff.br, alexbra1@gmail.com
}

\medskip

\noindent{\small Marco Pacini \\
Universidade Federal Fluminense, Rua Prof. M. W. de Freitas, S/N\\ 
Niter\'oi, Rio de Janeiro, Brazil. 24210-201.\\
email: marco\_pacini@id.uff.br, pacini.uff@gmail.com
}

\medskip

\noindent{\small Matheus Secco \\
Departamento de Matem\'atica, PUC-Rio, Rua Marqu\^es de S\~ao Vicente 225\\
Rio de Janeiro, Rio de Janeiro, Brazil. 22451-900. \\
email: matheussecco@mat.puc-rio.br, matheussecco@gmail.com}

\end{document}